\documentclass[12pt,leqno]{amsart}
\usepackage{verbatim}
\usepackage{amsthm}
\usepackage{amstext}
\usepackage{amssymb}
\usepackage[dvipdfmx]{graphicx}
\usepackage{mathrsfs}
\usepackage{amscd}

\makeatletter
\numberwithin{equation}{section}
\numberwithin{figure}{section}

\theoremstyle{plain}
\newtheorem{theorem}{Theorem}[section]
\newtheorem{lemma}[theorem]{Lemma}

\newtheorem{corollary}[theorem]{Corollary}
\newtheorem{proposition}[theorem]{Proposition}
\newtheorem{claim}[theorem]{Claim}

\theoremstyle{definition}

\newtheorem{remark}[theorem]{Remark}
\newtheorem{example}[theorem]{Example}

\newtheorem{condition}[theorem]{Condition}
\newcommand{\mdim}{\mathrm{mdim}}
\newcommand{\diam}{\mathrm{diam}}

\newcommand{\norm}[1]{\left|\!\left|#1\right|\!\right|}

\makeatother

\begin{document}

\title[From rate distortion theory to metric mean dimension]{From rate distortion theory to metric mean dimension: variational principle}

\author{Elon Lindenstrauss, Masaki Tsukamoto}



\date{\today}

\thanks{We gratefully acknowledge the following sources of support: E.L. was supported by the European Research Council (Advanced Research Grant
267259) and the ISF (891/15). M.T. stay at the Hebrew University was supported by the John Mung Program of Kyoto University.}

\maketitle

\begin{abstract}
The purpose of this paper is to point out a new connection between information theory and dynamical systems.
In the information theory side, we consider rate distortion theory, which studies lossy data compression of 
stochastic processes under distortion constraints.
In the dynamical systems side, we consider mean dimension theory, which studies how many parameters per second we need
to describe a dynamical system. 
The main results are new variational principles connecting rate distortion function to metric mean dimension.
\end{abstract}

\section{Introduction} \label{section: introduction}

\subsection{Main results}  \label{subsection: main result}

There is a long tradition in the study of dynamical systems to consider 
the interplay between ergodic theory and topological dynamics (see e.g.\ Glasner--Weiss~\cite{Glasner--Weiss} for an in depth discussion).
An important manifastation of this interplay is the \textit{variational principle} relating measure theoretic and topological entropy
(Goodwyn \cite{Goodwyn}, Dinaburg \cite{Dinaburg} and Goodman \cite{Goodman}).
Let $(\mathcal{X}, T)$ be a dynamical system, i.e.\ $\mathcal{X}$ is a compact metric space and $T$ is a continuous map
from $\mathcal{X}$ to $\mathcal{X}$.
We denote by $\mathscr{M}^T(\mathcal{X})$ the set of all invariant probability measures on $\mathcal{X}$.
The variational principle connects the 
measure theoretic entropy $h_\mu(T)$ to the topological entropy $h_{\mathrm{top}}(T)$ by 
\begin{equation} \label{eq: variational principle}
  h_{\mathrm{top}}(T) = \sup_{\mu\in \mathscr{M}^T(\mathcal{X})} h_\mu (T).
\end{equation}  

In the end of the last century Gromov \cite{Gromov} proposed a new topological invariant of dynamical 
systems called \textit{mean dimension}.
The mean dimension of a dynamical system $(\mathcal{X},T)$ is denoted by $\mdim(\mathcal{X},T)$.
This invariant counts the average number of parameters needed per itaration for describing a point in $\mathcal{X}$, and gives a non-degenerate numerical invariant
for dynamical systems of \textit{infinite dimensional} and \textit{infinite entropy}.

For example, consider the infinite product of the unit interval 
\[ [0,1]^\mathbb{Z} = \cdots \times [0,1]\times [0,1] \times [0,1] \times \cdots \]
and let $\sigma$ be the shift map on this space.
The system $([0,1]^\mathbb{Z},\sigma)$ is obviously infinite dimensional and has infinite topological entropy, but its mean dimension is one.
Intuitively this means that to describe an orbit in $([0,1]^\mathbb{Z}, \sigma)$ one needs one parameter per iterate.
This is analogous to the fact that the symbolic shift $\{1,2,\dots, n\}^\mathbb{Z}$ has the topological entropy 
$\log n$.

Mean dimension has applications to topological dynamics, which cannot be touched within the framework of topological entropy.
Here we briefly explain an application to a natural \textit{embedding problem}, raised many years before the definition of mean dimension: \begin{quotation}
\textit{When can we embed a dynamical system $(\mathcal{X}, T)$ in the shift $([0,1]^\mathbb{Z}, \sigma)$?}
\end{quotation}
Mean dimension provides a necessary condition: If $(\mathcal{X},T)$ is embeddable in $([0,1]^\mathbb{Z},\sigma)$
then $\mdim(\mathcal{X},T)\leq 1$.
A deeper result \cite[Theorem 1.4]{Gutman--Tsukamoto} 
states that a minimal system $(\mathcal{X},T)$ of mean dimension less than $1/2$
can be embedded into $([0,1]^\mathbb{Z},\sigma)$ (this strenghened \cite[Theorem 5.1]{Lindenstrauss}, which proved a similar result but with a non-optimal constant). 
The result \cite[Theorem 1.4]{Gutman--Tsukamoto} is optimal in the sense that there exists a minimal system of mean dimension $1/2$ which cannot be embedded in
$([0,1]^\mathbb{Z},\sigma)$ (\cite[Theorem 1.3]{Lindenstrauss--Tsukamoto}).
These results show that mean dimension is certainly a reasonable measure of the ``size'' of dynamical systems.
The theory of mean dimension turns out to have connection to problems in several different mathematical fields, e.g.\
topological dynamics (\cite{Lindenstrauss, Lindenstrauss--Weiss, Li, Gutman 1, Gutman 2, Gutman 3}), 
geometric analysis (\cite{Costa, Matsuo--Tsukamoto}), and operator algebra (\cite{Li--Liang, Elliott--Niu}).

Motivated by the success of the variational principle (\ref{eq: variational principle}), 
one might want to define a \textit{measure theoretic mean dimension} and try to prove a 
corresponding variational principle,
but any na{\"\i}ve attemp to carry out this idea is doomed to failure.
The reason can be easily seen by using the Jewett--Krieger theorem (\cite{Jewett, Krieger}):
Every ergodic measurable dynamical system has a uniquely ergodic model.
Consider an arbitrary ergodic measurable dynamical system $\mathscr{X}$. 
Suppose we want to define its ``measure theoretic mean dimension''.
There exists a topological system $(\mathcal{X}, T)$ such that it is uniquely ergodic (i.e.\ $\mathscr{M}^T(\mathcal{X})$ consists of a single measure, say $\mu$)
and $(\mathcal{X},\mu,T)$ is measurably isomorphic to $\mathscr{X}$.
It is known that uniquely ergodic systems always have zero topological mean dimension 
(\cite[Theorem 5.4]{Lindenstrauss--Weiss}).
Then if we have a ``variational principle'', the only possibility is that the ``measure theoretic mean dimension'' of $\mathscr{X}$ is zero.

It turned out that \textit{rate distortion theory} and \textit{metric mean dimension} provide a much better framework to study this interplay.
Rate distortion theory is a standard concept in information theory originally introduced by the monumental paper of 
Shannon \cite{Shannon}.
Its primary object is data compression of \textit{continuous} random variables and their processes.
Continuous random variables always have infinite entropy, so it is impossible to describe them perfectly with only finitely many bits.
Instead rate distortion theory studies a \textit{lossy} data compression method achieving some \textit{distortion} constraints.
A friendly introduction can be found in Cover--Thomas \cite[Chapter 10]{Cover--Thomas}.
Metric mean dimension is a metric space version of mean dimension 
introduced by Weiss and the first named author \cite{Lindenstrauss--Weiss} that is related to mean dimension in a way that is very analogous to how Minkowski or Hausdorff dimensions are related to the topological dimension.
Both rate distortion theory and metric mean dimension use \textit{distance} as a crucial ingredient.
This metric structure enables us to give a meaningful variational principle.

First we explain rate distortion theory.
For a couple $(X,Y)$ of random variables $X$ and $Y$ we denote its mutual information by $I(X;Y)$.
We review the definition and basic properties of $I(X;Y)$ in Section \ref{section: mutual information}.
Intuitively it is the amount of information which $X$ and $Y$ share.
Let $(\mathcal{X},T)$ be a dynamical system with a distance $d$ on $\mathcal{X}$.
Take an invariant probability measure $\mu\in \mathscr{M}^T(\mathcal{X})$. 
For a positive number $\varepsilon$ we define the \textbf{rate distortion function} $R_\mu(\varepsilon)$ as the 
infimum of
\begin{equation} \label{eq: definition of rate distortion function}
   \frac{I(X;Y)}{n}, 
\end{equation}   
where $n$ runs over all natural numbers, and $X$ and 
$Y = (Y_0,\dots,Y_{n-1})$ are random variables defined on some probability space $(\Omega, \mathbb{P})$ such that 
\begin{itemize}
  \item $X$ takes values in $\mathcal{X}$, and its law is given by $\mu$.
  \item Each $Y_k$ takes values in $\mathcal{X}$, and $Y$ approximates the process $(X, TX, \dots, T^{n-1}X)$ in the 
  sense that
  \begin{equation} \label{eq: distortion condition}
    \mathbb{E}\left(\frac{1}{n} \sum_{k=0}^{n-1} d(T^k X, Y_k) \right) < \varepsilon.
  \end{equation}  
\end{itemize}
Here $\mathbb{E}(\cdot)$ is the expectation with respect to the probability measure $\mathbb{P}$.
Note that $R_\mu(\varepsilon)$ depends on the distance $d$ although it is not explicitly written in the notation.

Roughly speaking, $R_\mu(\varepsilon)$ is the minimum rate of quantizations of the process $\{T^k X\}_{k=0}^\infty$
under the distortion constraint (\ref{eq: distortion condition}).
More precisely, a main theorem of rate distortion theory \cite[Chapter 7]{Berger} 
states that if the invariant measure $\mu$ is ergodic then there exists a sequence of 
maps $f_n = (f_{n,0}, \dots, f_{n,n-1}):\mathcal{X}\to \mathcal{X}^n$ $(n\geq 1)$ satisfying 
\[ \lim_{n\to \infty} \frac{\log |f_n(\mathcal{X})|}{n} = R_\mu(\varepsilon), \quad 
     \mathbb{E}\left(\frac{1}{n} \sum_{k=0}^{n-1} d(T^k X, f_{n,k}(X)) \right) < \varepsilon,   \]
where $X$ is a random variable obeying $\mu$ (and $| {f _ n (\mathcal X)}|$ denotes the cardinality of the set $f _ n (\mathcal X)$).     
Then we can represent the process $\{T^k X\}_{k=0}^\infty$ by the quantization
\begin{multline*} f_{n,0}(X), \dots, f_{n,n-1}(X), f_{n,0}(T^n X), \dots,\\
 f_{n,n-1}(T^n X), f_{n,0}(T^{2n}X), \dots, f_{n,n-1}(T^{2n}X), \dots.
 \end{multline*}
This approximates $\{T^k X\}_{k=0}^\infty$ by $\varepsilon$ in average, and (if $n$ is sufficiently large) we need 
\[ \frac{\log|f_n(\mathcal{X})|}{n} \approx R_\mu(\varepsilon) \text{ nats per second} \]
for describing the sequence\footnote{``nats'' means ``natural unit of information''. Here the 
base of the logarithm is $e$ not $2$.}.
There also exists a similar theorem for non-ergodic $\mu$, but the statement is a bit more involved.
See \cite{ECG, LDN} for the details.

Next we explain metric mean dimension.
Let $(\mathcal{X},T)$ be a dynamical system with a distance $d$ as above.
For a positive number $\varepsilon$ we define $\#(\mathcal{X},d,\varepsilon)$ as the minimum cardinarity $N$ of the open covering 
$\{U_1,\dots,U_N\}$ of $\mathcal{X}$ such that all $U_n$ have diameter smaller than $\varepsilon$.
For a natural number $n$ we define a distance $d_n$ on $\mathcal{X}$ by 
\begin{equation}  \label{eq: d_n distance}
   d_n(x,y) = \max_{0\leq k < n} d(T^k x, T^k y). 
\end{equation}
We set 
\[ S(\mathcal{X},T,d,\varepsilon) = \lim_{n \to \infty} \frac{\log \#(\mathcal{X}, d_n, \varepsilon)}{n}. \]
This limit always exists because $\log  \#(\mathcal{X}, d_n, \varepsilon)$ is a subadditive function of $n$.
 The topological entropy $h_\mathrm{top}(T)$ is the limit of $S(\mathcal{X},T,d,\varepsilon)$ as $\varepsilon \to 0$.
When the topological entropy is infinite, we are interested in the growth of $S(\mathcal{X},T,d, \varepsilon)$.
This motivates the definition of upper and lower \textbf{metric mean dimension}:
\begin{align*}
 \overline{\mdim}_\mathrm{M}(\mathcal{X},T,d) &= \limsup_{\varepsilon\to 0} \frac{S(\mathcal{X},T,d,\varepsilon)}{|\log \varepsilon|}, \\
    \underline{\mdim}_\mathrm{M}(\mathcal{X},T,d) &= \liminf_{\varepsilon\to 0} \frac{S(\mathcal{X},T,d,\varepsilon)}{|\log \varepsilon|}. 
\end{align*}
If the limit supremum and infimum agree, we denote the common value by $\mdim_M(\mathcal{X}, T,d)$.

By \cite[Theorem 4.2]{Lindenstrauss--Weiss} the metric mean dimensions always dominate the 
topological mean dimension:
\begin{equation} \label{eq: metric mean dimension dominates mean dimension}
    \mdim(\mathcal{X},T) \leq \underline{\mdim}_\mathrm{M}(\mathcal{X},T,d)  \leq \overline{\mdim}_\mathrm{M}(\mathcal{X},T,d). 
\end{equation}    
It is also known (\cite[Theorem 4.3]{Lindenstrauss}) that if $(\mathcal{X},T)$ is minimal then there exists a distance $d$ on 
$\mathcal{X}$ satisfying
\[ \underline{\mdim}_\mathrm{M}(\mathcal{X},T,d)= \mdim(\mathcal{X},T). \]
It is conjectured that such a distance exists for every system.

Metric mean dimension is not just a theoretical object. 
It is an important tool for computing topological mean dimension.
At least in our experience, it is generally difficult to prove upper bounds on topological mean dimension.
The most powerful method (known to the authors) 
is to use metric mean dimension.
If we obtain an upper bound on metric mean dimension, then we can also bound 
topological mean dimension by the inequality (\ref{eq: metric mean dimension dominates mean dimension}).
The papers \cite{Tsukamoto 1, Tsukamoto 2} employ this method to compute the topological mean dimensions of 
certain dynamical systems in geometric analysis and complex geometry.

The main purpose of this paper is to establish a variational principle connecting rate distortion function to 
metric mean dimension.
Before going further, we look at an example:

\begin{example}  \label{ex: Hilbert cube}
Let $\mathcal{X}=[0,1]^\mathbb{Z}$ be the infinite product of the unit interval, and 
let $T: \mathcal{X}\to \mathcal{X}$ be the shift: $T\left((x_m)_{m\in \mathbb{Z}}\right) = (x_{m+1})_{m\in \mathbb{Z}}$.
We define a distance $d$ on $\mathcal{X}$ by 
\begin{equation}  \label{eq: distance on Hilbert cube}
    d(x,y) = \sum_{m \in \mathbb{Z}} 2^{-|m|} |x_m-y_m|,    \quad 
    \left( x= (x_m)_{m\in \mathbb{Z}}, \, y= (y_m)_{m\in \mathbb{Z}} \right).
\end{equation}
First we calculate the metric mean dimension.
Let $\varepsilon>0$ and set $l =\lceil \log_2(4/\varepsilon)\rceil$.
Then $\sum_{|n|>l} 2^{-|n|} \leq \varepsilon/2$.
We consider an open covering of $[0,1]$ by
\[  I_k = \left( \frac{(k-1)\varepsilon}{12}, \frac{(k+1)\varepsilon}{12} \right), \quad 0\leq k\leq \lfloor 12/\varepsilon\rfloor. \]
$I_k$ has length $\varepsilon/6$.
For $n\geq 1$, consider
\[ [0,1]^\mathbb{Z} = \! \bigcup_{0\leq k_{-l}, \dots, k_{n+l}\leq \lfloor 12/\varepsilon\rfloor} \! 
   \left\{x |\, x_{-l} \in I_{k_{-l}}, x_{-l+1}\in I_{k_{-l+1}}, \dots, x_{n+l}\in I_{k_{n+l}}\right\}. \] 
Each open set in the right-hand side has diameter less than $\varepsilon$ with respect to the distance $d_n$.
Hence 
\begin{equation}  \label{eq: spanning number of Hilbert cube}
     \#([0,1]^\mathbb{Z}, d_n, \varepsilon) \leq \left(1 + \lfloor 12/\varepsilon\rfloor\right)^{n+2l+1}
     = \left(1 + \lfloor 12/\varepsilon\rfloor\right)^{n+2\lceil \log_2(4/\varepsilon)\rceil +1}. 
\end{equation}     
On the other hand, any two distinct points in the sets 
\[ \left\{x\in [0,1]^\mathbb{Z}|\,
     x_m\in \{0,\varepsilon, 2\varepsilon, \dots, \lfloor 1/\varepsilon\rfloor \varepsilon\} \text{ for all $0\leq m<n$} \right\}  \]
have distance $\geq \varepsilon$ with respect to $d_n$.
It follows $\#(\mathcal{X}, d_n, \varepsilon) \geq (1+\lfloor 1/\varepsilon\rfloor)^n$.
Therefore
\[ S(\mathcal{X},T,d,\varepsilon) = \lim_{n\to \infty} \frac{\log \#(\mathcal{X}, d_n, \varepsilon)}{n} \sim |\log \varepsilon|
    \quad (\varepsilon\to 0).  \]
Thus $\mdim_{\mathrm{M}} (\mathcal{X},T,d) = 1$.

Next we consider the rate distortion function for the measure $\mu = \left(\text{Lebesgue measure}\right)^{\otimes \mathbb{Z}}$.
The calculation of $R_\mu(\varepsilon)$ requires some familiarity with mutual information, so we postpone it to Example \ref{ex: continuation of Hilbert cube} in Section \ref{section: mutual information}, 
and here we state only the result:
\begin{equation}  \label{eq: rate distortion function of Hilbert cube}
   R_\mu(\varepsilon) \sim |\log \varepsilon| \quad (\varepsilon\to 0). 
\end{equation}   
Therefore 
\[ \lim_{\varepsilon\to 0} \frac{R_\mu(\varepsilon)}{|\log \varepsilon|} = 1 = \mdim_{\mathrm{M}}(\mathcal{X},T,d). \]
The purpose of this paper is to generalize this phenomena to arbitrary dynamical systems.
\end{example}

For some of our results, we need to introduce a certain regularity condition on the underlying mertic space.
\begin{condition}\label{condition: distance}
 Let $(\mathcal{X},d)$ be a compact metric space. It is said to have \emph{tame growth of covering numbers} if 
 for every $\delta>0$ we have 
 \[ \lim_{\varepsilon\to 0} \varepsilon^\delta \log\#(\mathcal{X},d,\varepsilon) = 0.\]
 Note that this is purely a condition on metric spaces and does not involve the dynamics.
\end{condition}
For example, if $\mathcal{X}$ is a compact subset of the Euclidean space $\mathbb{R}^n$, then
\[  \#(\mathcal{X}, \text{Euclidean distance},\varepsilon) = O((1/\varepsilon)^n), \]
and so $\mathcal{X}$ satisfies Condition \ref{condition: distance}.
Indeed the tame growth of covering numbers condition is a fairly mild condition:

\begin{lemma} \label{lemma: existence of distance}
 Every compact metrizable space admits a distance satisfying Condition~\ref{condition: distance}.
\end{lemma}
\begin{proof}
Every compact metrizable space can be topologically embedded into the infinite dimensional cube $[0,1]^\mathbb{Z}$,
so it is enough to prove the statement for $[0,1]^\mathbb{Z}$.
Let $d$ be the distance introduced in (\ref{eq: distance on Hilbert cube}).
By (\ref{eq: spanning number of Hilbert cube})
\[ \#([0,1]^\mathbb{Z},d,\varepsilon) \leq 
      \left(1 + \lfloor 12/\varepsilon\rfloor\right)^{2\lceil \log_2(4/\varepsilon)\rceil +2}.  \]
It follows 
\[ \log \#([0,1]^\mathbb{Z},d,\varepsilon)  = O\left(|\log\varepsilon|^2\right). \]
This satisfies  the tame growth of covering numbers condition.     
\end{proof}

\begin{remark} \label{remark: mildness of condition}
It is easy to check that
if $(A,d)$ is a compact metric space satisfying Condition \ref{condition: distance}
then the distance $d'$ on the shift $A^\mathbb{Z}$ defined by 
\[ d'(x,y) = \sum_{n\in \mathbb{Z}} 2^{-|n|} d(x_n,y_n) \]
also satisfies Condition \ref{condition: distance}.
\end{remark}

Our first main result is: 
\begin{theorem} \label{thm: main theorem}
Let $(\mathcal{X},T)$ be a dynamical system with a distance $d$.
Suppose $d$ satisfies Condition \ref{condition: distance}. Then 
\begin{equation} \label{eq: rate distortion variational principle}
  \begin{split}
    \overline{\mdim}_\mathrm{M}(\mathcal{X},T,d) &= 
     \limsup_{\varepsilon\to 0} \frac{\sup_{\mu\in \mathscr{M}^T(\mathcal{X})}R_\mu(\varepsilon)}{|\log\varepsilon|}, \\
     \underline{\mdim}_\mathrm{M}(\mathcal{X},T,d) &= 
     \liminf_{\varepsilon\to 0} \frac{\sup_{\mu\in \mathscr{M}^T(\mathcal{X})}R_\mu(\varepsilon)}{|\log\varepsilon|}.
   \end{split}
\end{equation}   
\end{theorem}
Therefore we can say that metric mean dimension is a topological dynamics counterpart of rate distortion theory.

\begin{remark}
Our formulation of the variational principle (\ref{eq: rate distortion variational principle})
is strongly influenced by the work of Kawabata--Dembo \cite{Kawabata--Dembo}.
For a metric space $A$, they studied connections between the fractal dimensions of $A$ and the rate distortion 
functions of i.i.d. processes taking values in $A$.
Theorem \ref{thm: main theorem} can be seen as a generalization of \cite[Proposition 3.1]{Kawabata--Dembo}
from the case of $(\mathcal{X},T) = (A^\mathbb{Z}, \mathrm{shift})$ to arbitrary dynamical systems.
\end{remark}

Although Condition \ref{condition: distance} is a mild condition, it might still look technical and one might want to remove it.
But indeed the equalities (\ref{eq: rate distortion variational principle}) do \textit{not} hold in general without an 
additional assumption:
\begin{proposition} \label{prop: counter example}
 There exists a dynamical system $(\mathcal{X}, T)$ with a distance $d$ such that 
\[ \mdim_\mathrm{M}(\mathcal{X}, T, d) = \infty, \quad 
   \lim_{\varepsilon\to 0}   \frac{\sup_{\mu\in \mathscr{M}^T(\mathcal{X})}R_\mu(\varepsilon)}{|\log\varepsilon|} = 0.\]
\end{proposition}

\begin{remark}  \label{remark: why technical condition}
In the proof of Theorem \ref{thm: main theorem}, we use
Condition \ref{condition: distance} to compare the two distances 
\begin{equation}  \label{eq: two distances}
   \frac{1}{n} \sum_{k=0}^{n-1} d(T^k x, T^k y)  \quad \text{and} \quad 
   \max_{0\leq k<n} d(T^k x, T^k y). 
\end{equation}   
The former is closely related to the distortion condition (\ref{eq: distortion condition}) in the definition of
rate distortion function. The latter is used in the definition of metric mean dimension.
Under Condition \ref{condition: distance}, these two distances behave quite similarly.
A rough idea of the proof of Proposition \ref{prop: counter example} is to construct a system $(\mathcal{X},T)$
where the two distances (\ref{eq: two distances}) show radically different behaviors.
\end{remark}

The above definition of the rate distortion function $R_\mu(\varepsilon)$, or the similar $L^2$-rate distortion function defined in \S\ref{subsection: L^p variants}, seems to be the most widely used one.
It has from our point of view the disadvantage that in this case we need to assume Condition \ref{condition: distance} for establishing the variational principle
(\ref{eq: rate distortion variational principle}).
Next we propose another version of rate distortion function and establish a corresponding variational principle without any additional condition.

Let $(\mathcal{X},T)$ be a dynamical system with a distance $d$ and an invariant probability measure $\mu$.
For positive numbers $\varepsilon$ and $\alpha$
we define the \textbf{$L^\infty$-rate distortion function} $\tilde{R}_\mu(\varepsilon, \alpha)$
as the infimum of 
\[ \frac{I(X;Y)}{n},   \]
where $n$ runs over all natural numbers,
and $X$ and $Y=(Y_0,\dots,Y_{n-1})$ are random variables defined on some probability space 
$(\Omega, \mathbb{P})$ such that 
\begin{itemize}
   \item $X$ takes values in $\mathcal{X}$, and its law is given by $\mu$.
   \item Each $Y_k$ takes values in $\mathcal{X}$, and they satisfy the following \textit{modified distortion condition}:
    \begin{equation}  \label{eq: modified distortion condition}
     \mathbb{E}\left(\text{the number of $k \in [0,n-1]$ satisfying $d(T^k X, Y_k)\geq \varepsilon$}\right)  < \alpha n. 
    \end{equation} 
\end{itemize}
   
\noindent
\medskip
In other words, we define $\tilde{R}_\mu(\varepsilon,\alpha)$ by replacing the distortion condition (\ref{eq: distortion condition})
in the definition of $R_\mu(\varepsilon)$ with (\ref{eq: modified distortion condition}).
We set 
\[  \tilde{R}_\mu(\varepsilon) = \lim_{\alpha\to 0} \tilde{R}_\mu(\varepsilon, \alpha). \]
The reason for our use of terminology ``$L^\infty$-rate distortion function'' will (hopefully) become clearer to the reader in the next subsection.

Our second main result is:

\begin{theorem}  \label{thm: second main theorem} 
For any dynamical system $(\mathcal{X},T)$ with a distance $d$, we have
\begin{equation} \label{eq: second rate distortion variational principle}
  \begin{split}
    \overline{\mdim}_\mathrm{M}(\mathcal{X},T,d) &= 
     \limsup_{\varepsilon\to 0} \frac{\sup_{\mu\in \mathscr{M}^T(\mathcal{X})}\tilde{R}_\mu(\varepsilon)}{|\log\varepsilon|}, \\
     \underline{\mdim}_\mathrm{M}(\mathcal{X},T,d) &= 
     \liminf_{\varepsilon\to 0} \frac{\sup_{\mu\in \mathscr{M}^T(\mathcal{X})}\tilde{R}_\mu(\varepsilon)}{|\log\varepsilon|}.
   \end{split}
\end{equation} 
\end{theorem}

\noindent
We emphasize that we do not need any additional condition for establishing (\ref{eq: second rate distortion variational principle}) in this case.

\subsection{$L^p$-variants} \label{subsection: L^p variants}

We can also consider $L^p$-versions of the variational principle.
The $L^2$-case might be of special interest because it is related to the least squares method.
Let $(\mathcal{X},T)$ be a dynamical system with a distance $d$.
For $1\leq p<\infty$, $\varepsilon>0$ and $\mu \in \mathscr{M}^T(\mathcal{X})$ 
we define the \textbf{$L^p$-rate distortion function} $R_{\mu,p}(\varepsilon)$ by replacing the distortion condition 
(\ref{eq: distortion condition}) in the definition of $R_\mu(\varepsilon)$ with 
\begin{equation}  \label{eq: L^p distortion condition}
    \mathbb{E}\left(\frac{1}{n}\sum_{k=0}^{n-1}d(T^k X, Y_k)^p\right) < \varepsilon^p. 
\end{equation}    
By the H\"older inequality, this is stronger than (\ref{eq: distortion condition}),
hence $R_\mu(\varepsilon) \leq R_{\mu,p}(\varepsilon)$.
On the other hand, the condition (\ref{eq: L^p distortion condition}) is essentially weaker than (\ref{eq: modified distortion condition})
in the definition of $\tilde{R}_\mu(\varepsilon)$. 
Indeed 
\begin{multline*} \frac{1}{n}\sum_{k=0}^{n-1}d(T^k X,Y_k)^p  \leq \\
\varepsilon^p + \left(\diam(\mathcal{X},d)\right)^p \cdot \frac{1}{n}
   \cdot |\{k\in [0, n-1]|\, d(T^k X, Y_k) \geq \varepsilon\}|.  
\end{multline*}
So the condition (\ref{eq: modified distortion condition}) implies  
\[    \mathbb{E}\left(\frac{1}{n}\sum_{k=0}^{n-1}d(T^k X,Y_k)^p\right)  < \varepsilon^p + \alpha  \left(\diam(\mathcal{X},d)\right)^p . \]
This leads to $R_{\mu,p}(\varepsilon') \leq \tilde{R}_\mu(\varepsilon)$ for any $\varepsilon'>\varepsilon$.
Thus we get 
\[  R_\mu(\varepsilon') \leq R_{\mu,p}(\varepsilon') \leq \tilde{R}_\mu(\varepsilon) \quad \text{for any $\varepsilon'>\varepsilon>0$}. \] 
Therefore Theorems \ref{thm: main theorem} and \ref{thm: second main theorem} imply

\begin{corollary}
If the distance $d$ satisfies Condition \ref{condition: distance}, then for any $p\geq 1$
 \begin{equation*} 
  \begin{split}
    \overline{\mdim}_\mathrm{M}(\mathcal{X},T,d) &= 
     \limsup_{\varepsilon\to 0} \frac{\sup_{\mu\in \mathscr{M}^T(\mathcal{X})}R_{\mu,p}(\varepsilon)}{|\log\varepsilon|}, \\
     \underline{\mdim}_\mathrm{M}(\mathcal{X},T,d) &= 
     \liminf_{\varepsilon\to 0} \frac{\sup_{\mu\in \mathscr{M}^T(\mathcal{X})}R_{\mu,p}(\varepsilon)}{|\log\varepsilon|}.
   \end{split}
\end{equation*}   
\end{corollary}


\subsection{Comments on the proofs and the organization of the paper} \label{subsection: comments}

The uniform distribution on the set $\{1,2,\dots,n\}$ has entropy $\log n$, and this is the maximal entropy measure 
among all probability distributions on it.
There exists a similar result about mutual information $I(X;Y)$:
Roughly speaking, if $X$ is uniformly distributed over an $\varepsilon$-separated set $S$ of a compact metric space~$\mathcal{X}$, and if~$\varepsilon^{-1}{\mathbb{E}\,\left(d(X,Y)\right)}$ is sufficiently small, then 
$I(X;Y)$ is almost equal to $\log |S|$
(for precise statements, see Corollary \ref{cor: uniform distribution over separated set} and
Lemma \ref{lemma: uniform distribution over separated set, second form} below).
This observation is key to the proofs of Theorems \ref{thm: main theorem} and \ref{thm: second main theorem}. 
Starting from this, we will follow a line of ideas analogous to 
Misiurewicz's proof \cite{Misiurewicz} of the variational principle (\ref{eq: variational principle}).
Misiurewicz's argument adapts quite naturally (perhaps even suprisingly so)
to the setting of rate distortion theory.

Organization of the paper is as follows: 
We recall some basics of mutual information in Section \ref{section: mutual information}.
Theorems \ref{thm: main theorem} and \ref{thm: second main theorem} are proved in Sections \ref{section: proof of main theorem}
and \ref{section: proof of second main theorem} respectively.
We prove Proposition~\ref{prop: counter example} in Section~\ref{section: proof of Proposition}.
We recall some elementary results on optimal transport (which are used in Sections~\ref{section: proof of main theorem}
and \ref{section: proof of second main theorem}) in the Appendix.

\vspace{0.2cm}

\textbf{Acknowledgement.}
We would like to thank Professor Benjamin Weiss for helpful discussions. The second named author would like to thank Professors Kazumasa Kuwada and Shinichi Ohta for advice on optimal transport.
This paper was written while the sescond named author stayed in the Einstein Institute of Mathematics in the Hebrew University of 
Jerusalem.
He would like to thank all the Institute staff for their hospitality.

\section{Mutual information} \label{section: mutual information}

In this section we recall some basic properties of mutual information.
A good reference is Cover--Thomas \cite[Chapter 2]{Cover--Thomas}.

Throughout this section $(\Omega,\mathbb{P})$ is a probability space.
Let $\mathcal{X}$ and $\mathcal{Y}$ be measurable spaces, and 
$X:\Omega\to \mathcal{X}$ and $Y:\Omega\to \mathcal{Y}$ measurable maps.
We define the \textbf{mutual information} $I(X;Y)$ as the supremum of 
\begin{equation} \label{eq: definition of mutual information}
   \sum_{m=1}^M \sum_{n=1}^N  \mathbb{P}\left((X,Y)\in P_m\times Q_n\right) 
   \log\frac{\mathbb{P}\left((X,Y)\in P_m\times Q_n\right)}{\mathbb{P}(X\in P_m) \mathbb{P}(Y\in Q_n)}, 
\end{equation}   
where $\{P_1,\dots,P_M\}$ and $\{Q_1,\dots,Q_N\}$ are partitions of $\mathcal{X}$ and $\mathcal{Y}$ respectively, with the convention that $0\log (0/a) =0$ for all $a\geq 0$.
The mutual information $I(X;Y)$ is nonnegative and symmetric: $I(X;Y) = I(Y;X) \geq 0$.

If $\mathcal{X}$ and $\mathcal{Y}$ are finite sets, then 
\begin{equation}  \label{eq: mutual information of discrete random variables}
   \begin{split}
    I(X;Y) &= \sum_{x\in \mathcal{X},\, y\in \mathcal{Y}} \mathbb{P}(X=x, Y=y) 
                \log \frac{\mathbb{P}(X=x,Y=y)}{\mathbb{P}(X=x) \mathbb{P}(Y=y)} \\
            &= H(X)-H(X|Y) = H(X)+H(Y)-H(X,Y), 
    \end{split}
\end{equation}            
where $H(X|Y)$ is the conditional entropy of $X$ given $Y$.
The formula $I(X;Y) = H(X)-H(X|Y)$ shows an intuitive meaning of mutual information;
it is the amount of information which the random variables $X$ and $Y$ share.

The following two lemmas are trivial but important in the proofs of the main theorems.

\begin{lemma}  \label{lemma: convergence of mutual information}
Suppose $\mathcal{X}$ and $\mathcal{Y}$ are finite sets.
Let $(X_n,Y_n):\Omega \to \mathcal{X}\times \mathcal{Y}$ $(n\geq 1)$ be a sequence of measurable maps converging to
$(X,Y):\Omega\to \mathcal{X}\times \mathcal{Y}$ in law.
Then $I(X_n;Y_n)$ converges to $I(X;Y)$.
\end{lemma}

\begin{proof}
This follows from the first equation of (\ref{eq: mutual information of discrete random variables}).
\end{proof}

\begin{lemma}[Data-processing inequality] \label{lemma: data-processing inequality}
Let $\mathcal{X},\mathcal{Y},\mathcal{Z}$ be measurable spaces, and
$X:\Omega\to \mathcal{X}$ and $Y:\Omega\to \mathcal{Y}$ measurable maps.
Let $f:\mathcal{Y}\to \mathcal{Z}$ be a measurable map. Then\footnote{Indeed data-processing inequality is a more general 
statement; see \cite[Section 2.8]{Cover--Thomas}. But we need only this statement here.} 
\[ I(X; f(Y)) \leq I(X;Y). \]
\end{lemma}
\begin{proof}
This immediately follows from the definition of $I(X;Y)$.
\end{proof}

\begin{remark}  \label{remark: reduction to discrete case}
Lemma \ref{lemma: data-processing inequality} implies that, 
in the definition \eqref{eq: definition of rate distortion function} of the rate distortion function $R_\mu(\varepsilon)$, 
we can assume that 
the random variable $Y$ there takes only finitely many values, namely that its distribution is supported on a finite set. 
Indeed, let $X$ and $Y$ be as in~\eqref{eq: definition of rate distortion function} and~\eqref{eq: distortion condition}.
Take a sufficiently fine partition $\mathcal{P}$ of $\mathcal{X}$ and
for each atom of $A$ of $\mathcal{P}$ choose one point $p_A \in A$.
Define $f:\mathcal{X}\to \mathcal{X}$ by $f(A)= \{p_A\}$, and set $Z = (Z_0,\dots,Z_{n-1}) = (f(Y_0),\dots,f(Y_{n-1}))$. Then 
\begin{align*} \mathbb{E}\left(\frac{1}{n}\sum_{k=0}^{n-1}d\left(T^k X, Z_k \right) \right)
  & \leq \max_{A\in \mathcal{P}} \diam (A) + \mathbb{E}\left(\frac{1}{n}\sum_{k=0}^{n-1}d\left(T^k X, Y_k\right) \right) \\
  &< \varepsilon
 \end{align*}
if $\mathcal{P}$ is sufficiently fine.   Hence $Z$ satisfies the distortion condition (\ref{eq: distortion condition}).
Lemma \ref{lemma: data-processing inequality} implies 
\[ I(X;Z) \leq I(X;Y). \]
The random variable $Z$ obviously takes only finitely many values.

Similarly we can also assume that $Y$ takes only finitely many values in the definition of $\tilde{R}_\mu(\varepsilon,\alpha)$:
Suppose $Y$ satisfies the modified distortion condition (\ref{eq: modified distortion condition}).
Then we can find $0<\varepsilon'<\varepsilon$ satisfying 
\[ \mathbb{E}\left(\text{number of $0\leq k \leq n-1$ satisfying $d(T^k X,Y_k)\geq \varepsilon'$}\right) <\alpha n.\]
If the partition $\mathcal{P}$ is sufficiently fine, then for $Z_k$ as above
\begin{multline*}
    \mathbb{E}\left(\text{number of $0\leq k \leq n-1$ satisfying $d(T^k X, Z_k)\geq \varepsilon$}\right) \\
	\begin{aligned}
    &\leq   \mathbb{E}\left(\text{number of $0\leq k \leq n-1$ satisfying $d(T^k X,Y_k)\geq \varepsilon'$}\right)\\
    &<\alpha n.
   \end{aligned}
\end{multline*}
\end{remark}

\noindent
For real numbers $0\leq p\leq 1$ we set $H(p) = -p\log p -(1-p)\log (1-p)$ (with $H(0)=H(1)=0$). 

\begin{lemma}[Fano's inequality]  \label{lemma: Fano's inequality}
Suppose $\mathcal{X}$, $\mathcal{Y}$ and $\mathcal{Z}$ are finite sets.
Let $f:\mathcal{Y}\to \mathcal{Z}$ be a map, and let $X:\Omega\to \mathcal{X}$ and $Y:\Omega\to \mathcal{Y}$ be measurable maps.
Set $P_e = \mathbb{P}(X\neq f(Y))$ (the probability of error).  Then\footnote{As in the case of data-processing inequality, 
Fano's inequality is more general than this statement; see \cite[Section 2.10]{Cover--Thomas}.}
\[  H(X|Y) \leq H(P_e) + P_e \log |\mathcal{X}|. \]
\end{lemma}

\begin{proof}
We briefly explain the proof given by \cite[Section 2.10]{Cover--Thomas} for the convenience of readers.
We define a random variable $E$ by 
\[ E= 0 \text{ if $X=f(Y)$}, \quad E=1 \text{ if $X\neq f(Y)$}. \]
We expand $H(X, E|Y)$ in two ways:
\begin{equation*}
  \begin{split}
  H(X,E|Y) &= H(X|Y) + H(E| X, Y) \\
                 &= H(E|Y) + H(X|E, Y).
  \end{split}
\end{equation*}  
We have $H(E|X, Y)=0$ because $E$ is determined by $X$ and $Y$. Thus 
\begin{equation*}
   \begin{split} 
     H(X|Y) &= H(E|Y) + H(X|E,Y)  \\ 
     &\leq H(E) + \mathbb{P}(E=0) H(X|E=0, Y) + \mathbb{P}(E=1) H(X|E=1, Y). 
    \end{split}
\end{equation*}     
It follows from the definition of $E$ that $H(E) = H(P_e)$ and $H(X|E=0, Y)=0$ (because $E=0$ means that 
$X$ is determined by $Y$).    
Since $X$ takes at most $|\mathcal{X}|$ values, $H(X|E=1, Y)\leq H(X)\leq \log |\mathcal{X}|$. 
Thus 
\[ H(X|Y) \leq H(P_e) + P_e\cdot  H(X|E=1, Y) \leq H(P_e) + P_e \log |\mathcal{X}|. \]
\end{proof}

The next corollary is essentially contained in \cite[Corollary A.1]{Kawabata--Dembo}.
This is the basis of the proof of Theorem \ref{thm: main theorem}.

\begin{corollary} \label{cor: uniform distribution over separated set}
Let $(\mathcal{X},d)$ be a compact metric space.
Let $\varepsilon >0$ and $D>2$.
Suppose $S\subset \mathcal{X}$ is a $(2D\varepsilon)$-separated set 
(i.e. any two distinct points in $S$ have distance $\geq 2D\varepsilon$).
Let $X$ and $Y$ be measurable maps from $\Omega$ to $\mathcal{X}$ such that $X$ is uniformly distributed over $S$
and 
\[ \mathbb{E}\left(d(X,Y)\right) < \varepsilon. \]
Then 
\[   I(X;Y) \geq \left(1-\frac{1}{D}\right) \log |S| - H(1/D). \]
\end{corollary}

\begin{proof}
Since $S$ is a finite set, $X$ takes only finitely many values.
We can assume that $Y$ also takes only finitely many values as in Remark \ref{remark: reduction to discrete case}.
Define $f:\mathcal{X}\to \mathcal{X}$ by 
\[ f(x) = \begin{cases}
a &  \text{if $x \in B_{D\varepsilon}(a)$ for some $a\in S$}, \\
    x & \text{otherwise},
\end{cases}
\]
with $B_{r}(x)$ denoting the open ball of radius $r$ around a point $x\in X$.
Set $P_e = \mathbb{P}(X\neq f(Y))$.
Since $\{X\neq f(Y)\}$ is contained in $\{d(X,Y) \geq D\varepsilon \}$, 
\[ P_e \leq \mathbb{P}\left(d(X,Y) \geq D\varepsilon \right) 
      \leq \frac{1}{D\varepsilon} \mathbb{E}\left(d(X,Y)\right) < \frac{1}{D} < \frac{1}{2}. \]
By Lemma \ref{lemma: Fano's inequality},
\[ H(X|Y) \leq H(P_e) + P_e \log |S|
   \leq H(1/D) + (1/D) \log |S|. \]
Since $X$ is uniformly distributed over $S$, its entropy is $\log |S|$.
Thus 
\begin{align*}
I(X;Y) &= H(X) -H(X|Y) \\
&= \log |S| - H(X|Y) \geq \left(1-\frac{1}{D}\right) \log |S| - H(1/D). 
\end{align*}
\end{proof}

\noindent
The next lemma is used in the proof of Theorem \ref{thm: second main theorem}.

\begin{lemma} \label{lemma: uniform distribution over separated set, second form}
 Let $(\mathcal{X},d)$ be a compact metric space with a finite subset~$A$.
 Let $n$ be a natural number and $\varepsilon, \alpha$ positive numbers with $\alpha \leq 1/2$.
 Suppose $S\subset A^n$ is a $2\varepsilon$-separated set with respect to the distance 
 \[ d_n\left((x_0, \dots,x_{n-1}), (y_0,\dots, y_{n-1})\right) = 
    \max_{0\leq k \leq n-1} d(x_k, y_k). \]
 Let $X=(X_0,\dots,X_{n-1})$ and $Y= (Y_0,\dots,Y_{n-1})$ be measurable maps from $\Omega$ to $\mathcal{X}^n$
 such that $X$ is uniformly distributed over $S$ and 
\begin{equation} \label{eq: modified distortion condition in preliminary}
  \mathbb{E}(\text{number of $k \in [0, n-1]$ satisfying $d(X_k,Y_k)\geq \varepsilon$}) < \alpha n. 
\end{equation}  
Then 
\[ I(X;Y) \geq \log |S| - n H(\alpha) -\alpha n \log |A|. \]  
\end{lemma}   

\begin{proof}
The argument is similar to the proof of Fano's inequality.
We can assume that $Y$ takes only finitely many values as in Remark \ref{remark: reduction to discrete case}.
We define a random variable $Z$ by 
\[ Z= \{k \in [0,n-1]|\, d(X_k,Y_k)\geq \varepsilon\} \subset \{0,1,\dots,n-1\}. \]
Note that by assumption (\ref{eq: modified distortion condition in preliminary}) we have that
$\mathbb{E}|Z| <\alpha n$.

\begin{claim}  \label{claim: entropy of error set}
\[  H(Z) \leq n H(\alpha). \]
\end{claim}

\begin{proof}
We define $Z_k$ $(0\leq k \leq n-1)$ by 
\[ Z_k = 0 \text{ if $k \not\in Z$}, \quad Z_k = 1 \text{ if $k \in Z$}. \]
We have $|Z| = Z_0+\dots+Z_{n-1}$ and $H(Z) = H(Z_0,\dots, Z_{n-1}) \leq H(Z_0)+\dots+H(Z_{n-1})$.
Set $\alpha_k = \mathbb{P}(Z_k =1)$.
From the concavity of $H(p) = -p\log p - (1-p)\log (1-p)$, 
\[ H(Z) \leq \sum_{k=0}^{n-1} H(\alpha_k) \leq n H\left(\frac{1}{n}\sum_{k=0}^{n-1} \alpha_k \right)
   \leq n H(\alpha), \]
   where we used $\sum \alpha_k = E|Z| < \alpha n$ and $\alpha\leq 1/2$.
\end{proof}

Expanding $H(X,Z|Y)$ in two ways:
\begin{equation*}
   \begin{split}
   H(X,Z|Y)  &= H(X|Y) + H(Z|X,Y) \\
               &= H(Z|Y) + H(X|Y,Z).
   \end{split}
\end{equation*}   
We have $H(Z|X,Y) =0$ because $Z$ is determined by $X$ and $Y$.
Hence by Claim \ref{claim: entropy of error set}
\begin{equation}  \label{eq: bounding H(X|Y)}
   H(X|Y) = H(X|Y,Z) + H(Z|Y) \leq H(X|Y,Z) + n H(\alpha).
\end{equation}
Take a subset $E\subset \{0,1,\dots,n-1\}$.
(We write $E^c = \{0,1,\dots,n-1\}\setminus E$.)
We estimate the conditional entropy $H(X|Y,Z=E)$.
Under the condition $Z=E$, we have $\max_{k \in E^c} d(X_k, Y_k) < \varepsilon$.
Since $S$ is $2\varepsilon$-separated with respect to $d_n$, 
for each $a\in \mathcal{X}^n$ the number of $x\in S$ satisfying 
\[ \max_{k \in E^c} d(x_k, a_k) < \varepsilon \]
is at most $|A|^{|E|}$.
Therefore the number of possible outcomes of $X$ (given $Y$ and $Z=E$) is at most $|A|^{|E|}$.
Thus 
\[  H\left(X| Y, Z=E \right) \leq |E|\log |A|. \]
It follows that
\begin{equation*}
    \begin{split}
     H(X|Y,Z) &= \sum_{E} \mathbb{P}(Z=E) H(X|Y,Z=E) \\
                &\leq \log |A| \sum_{E} |E|\cdot  \mathbb{P}(Z=E) \\
                &= \log |A| \cdot \mathbb{E}|Z| \\
                &\leq \alpha n \log |A| \quad (\text{by the assumption $\mathbb{E}|Z| < \alpha n$}).    
    \end{split}
\end{equation*}    
Combining (\ref{eq: bounding H(X|Y)})
\begin{equation*}
    I(X;Y) = H(X) - H(X|Y)  \geq \log |S| - n H(\alpha) - \alpha n \log |A|.
\end{equation*}    
Here we used $H(X)= \log |S|$ since $X$ is uniformly distributed over~$S$.
\end{proof}

In the rest of this section we assume for simplicity that $\mathcal{X}, \mathcal{Y}, \mathcal{Z}$ are finites sets.

\begin{lemma}[Subadditivity of mutual information]  \label{lemma: subadditivity of mutual information}
Let $X, Y, Z$ be measurable maps from $\Omega$ to $\mathcal{X}, \mathcal{Y},\mathcal{Z}$ respectively.
Suppose $X$ and $Z$ are conditionally independent given $Y$, namely for every $y\in \mathcal{Y}$ with 
$\mathbb{P}(Y=y)\neq 0$ we have 
\begin{equation} \label{eq: conditional independence}
  \mathbb{P}(X=x,Z=z|Y=y) = \mathbb{P}(X=x|Y=y)\mathbb{P}(Z=z|Y=y)  
\end{equation}  
for every $x\in \mathcal{X}$ and $z\in \mathcal{Z}$. 
Then 
\[ I(Y; X, Z) \leq I(Y;X) + I(Y;Z). \]
\end{lemma}

\begin{proof}
From the conditional independence,
\begin{equation} \label{eq: equality of conditional entropies}
  H(X,Z|Y) = H(X|Y) + H(Z|Y). 
\end{equation}  
Indeed $H(X,Z|Y)$ is equal to 
\[ - \sum_y \mathbb{P}(Y=y)\left( \sum_{x,z} \mathbb{P}(X=x, Z=z|Y=y)\log  \mathbb{P}(X=x, Z=z|Y=y) \right).   \]
By using (\ref{eq: conditional independence}) we can easily check (\ref{eq: equality of conditional entropies}).
Then 
\begin{equation*}
   \begin{split}
    I(Y; X,Z)  &= H(X,Z) - H(X,Z|Y) \\
	&= H(X,Z) - H(X|Y) -H(Z|Y)  
	\\
                &\leq H(X) + H(Z) - H(X|Y) -H(Z|Y)\\
&				= I(X;Y) + I(Z;Y). 
   \end{split}
\end{equation*}
In the passage from the second line to the third, we used $H(X,Z) \leq H(X)+H(Z)$.
\end{proof}

On the other hand, we have:
\begin{lemma}[Superadditivity of mutual information]  \label{lemma:superadditivity of mutual information}
Let $X, Y, Z$ be measurable maps from $\Omega$ to $\mathcal{X}, \mathcal{Y},\mathcal{Z}$ respectively.
Suppose $X$ and $Z$ are independent.
Then
\[ I(Y; X, Z) \geq I(Y;X) + I(Y;Z). \]
\end{lemma}

\begin{proof}
Since $X, Z$ are independent, $H (X,Z)= H (X)+ H (Z)$ hence
\begin{align*}
I(Y;X,Z)& =H(X,Z)-H(X,Z|Y) = H (X) + H (Z) - H (X, Z| Y) \\
& \geq H (X) + H (Z) - H (X| Y) - H (Z| Y) \\
& = I(Y;X)+I(Y;Z)
.\end{align*}
\end{proof}

Let $X:\Omega\to \mathcal{X}$ and $Y:\Omega\to \mathcal{Y}$ be measurable maps.
We define a probability mass function $\mu(x)$ and a conditional probability mass function 
$\nu(y|x)$ by 
\[ \mu(x) = \mathbb{P}(X=x), \quad \nu(y|x) = \mathbb{P}(Y=y|X=x). \]
Notice that $\nu(y|x)$ is defined only for $x\in \mathcal{X}$ with 
$\mathbb{P}(X=x) \neq 0$.
The distribution of $(X,Y)$ is given by $\mu(x)\nu(y|x)$ and it determines the mutual information $I(X;Y)$, hence we sometimes write $I(X;Y) = I(\mu,\nu)$.

\begin{lemma}[Concavity/convexity of mutual information]  \label{lemma: concavity/convexity of mutual information}
$I(\mu, \nu)$ is a concave function of $\mu(x)$ for fixed $\nu(y|x)$ and a convex function of $\nu(y|x)$ for fixed $\mu(x)$.
More precisely,

\begin{enumerate}
\item Suppose that for each $x\in \mathcal{X}$ we are given a probability mass function $\nu(\cdot|x)$ on $\mathcal{Y}$.
Let $\mu_1$ and $\mu_2$ be two probability mass functions on $\mathcal{X}$. Then 
\[ I((1-t)\mu_1+t\mu_2, \nu) \geq (1-t)I(\mu_1,\nu) + t I(\mu_2,\nu) \quad (0\leq t\leq 1).\]
Here the left-hand side is the mutual information of the joint distribution $(1-t)\mu_1(x)\nu(y|x) + t\mu_2(x)\nu(y|x)$.
\item Suppose that for each $x\in \mathcal{X}$ we are given two probability mass functions 
$\nu_1(\cdot|x)$ and $\nu_2(\cdot|x)$ on $\mathcal{Y}$. 
Let $\mu$ be a probability mass function on $\mathcal{X}$. Then 
\[  I(\mu, (1-t)\nu_1+t\nu_2) \leq (1-t) I(\mu, \nu_1) + t I(\mu, \nu_2) \quad (0\leq t\leq 1).\]
Here the left-hand side is the mutual information of the joint distribution 
$(1-t) \mu(x)\nu_1(y|x) + t\mu(x)\nu_2(y|x)$.
\end{enumerate}
\end{lemma}

\begin{proof}
See \cite[Theorem 2.7.4]{Cover--Thomas} for the detailed proof.
Here we sketch the outline.
First we explain (1).
\[ I(\mu,\nu) = I(X;Y) = H(Y) -H(Y|X). \]
If $\nu(y|x)$ is fixed, $H(Y)$ is a concave function of $\mu(x)$ and $H(Y|X)$ is a linear function of $\mu(x)$.
The difference $I(\mu,\nu)$ is a concave function of $\mu(x)$.

Next we explain (2).
The function $\phi(t) = t\log t$ is convex. So 
\[ \phi\left(\frac{a+a'}{b+b'}\right) \leq \frac{b}{b+b'}\phi\left(\frac{a}{b}\right) + \frac{b'}{b+b'}\phi\left(\frac{a'}{b'}\right) \] 
for positive numbers $a,a',b,b'$. This leads to
\begin{equation}  \label{eq: log sum inequality}
   (a+a') \log \frac{a+a'}{b+b'} \leq a\log \frac{a}{b}  + a' \log \frac{a'}{b'}. 
\end{equation} 
Set $\sigma_i (y) = \sum_{x\in \mathcal{X}}\mu(x)\nu_i(y|x)$ for $i=1,2$.
Then $I(\mu, (1-t)\nu_1+ t\nu_2)$ is given by 
\[ \sum_{x,y} \left\{(1-t)\mu(x)\nu_1(y|x) + t\mu(x) \nu_2(y|x)\right\}
   \log \frac{(1-t)\mu(x)\nu_1(y|x) + t\mu(x) \nu_2(y|x)}{(1-t)\mu(x)\sigma_1(y) + t\mu(x)\sigma_2(y)}. \]
Applying the inequality (\ref{eq: log sum inequality}) to each summand, $I(\mu, (1-t)\nu_1+t\nu_2)$ is bounded by 
\[ \sum_{x,y} (1-t)\mu(x)\nu_1(y|x) \log \frac{\mu(x)\nu_1(y|x)}{\mu(x)\sigma_1(y)}
   + \sum_{x,y} t\mu(x)\nu_2(y|x) \log\frac{\mu(x)\nu_2(y|x)}{\mu(x)\sigma_2(y)}.   \]
  This is equal to $(1-t) I(\mu,\nu_1) + t I(\mu,\nu_2)$.
\end{proof}

\begin{example}[Continuation of Example \ref{ex: Hilbert cube}]  \label{ex: continuation of Hilbert cube}
Here we sketch the proof of the estimate (\ref{eq: rate distortion function of Hilbert cube}) in Example \ref{ex: Hilbert cube}.
Note that this is not used for the proofs of Theorems \ref{thm: main theorem} and \ref{thm: second main theorem}.
We use the notations in Example \ref{ex: Hilbert cube}.
It is easy to prove 
\[  \limsup_{\varepsilon\to 0}\frac{R_\mu(\varepsilon)}{|\log\varepsilon|} \leq 1. \]
See Lemma \ref{lemma: metric mean dimension dominates rate distortion function} below for the details.
The main issue is a lower bound on $R_\mu(\varepsilon)$.
Let $X$ and $Y= (Y_0,\dots,Y_{n-1})$ be random variables defined on some probability space such that $X$ has distribution $\mu$ and 
$Y_k$ take values in $[0,1]^\mathbb{Z}$ satisfying the distortion condition (\ref{eq: distortion condition}).
We write $X= (X_m)_{m\in \mathbb{Z}}$ and $Y_k = (Y_{k,m})_{m\in \mathbb{Z}}$.
\begin{equation}  \label{eq: lower bound on mutual information in the case of Hilbert cube}
   \begin{split}
    I(X;Y) &\geq I\left( (X_0,\dots, X_{n-1}); (Y_{0,0}, Y_{1,0},\dots, Y_{n-1,0})\right)  \\
   &  \hspace{1cm}  (\text{by data-processing inequality; see Lemma \ref{lemma: data-processing inequality}}) \\
   & \geq  \sum_{m=0}^{n-1} I\left(X_m;  (Y_{0,0}, Y_{1,0},\dots, Y_{n-1,0})\right) \\
   &  \hspace{1cm}   (\text{since $X_0,\dots, X_{n-1}$ are independent; see Lemma~\ref{lemma:superadditivity of mutual information}})  \\
   & \geq \sum_{m=0}^{n-1} I(X_m; Y_{m,0}) \quad 
      (\text{by data-processing inequality}).
    \end{split}
\end{equation}    
It follows from the distortion condition (\ref{eq: distortion condition}) that 
\begin{equation}  \label{eq: distortion in the case of Hilbert cube}
   \frac{1}{n}\sum_{m=0}^{n-1}\mathbb{E}|X_m-Y_{m,0}| \leq 
    \frac{1}{n} \mathbb{E}\left(\sum_{m=0}^{n-1} d(T^m X, Y_m) \right) < \varepsilon.   
\end{equation}    
We denote by $r(\varepsilon)$ the infimum of the mutual information $I(U;V)$ 
such that $U$ and $V$ are random variables (defined on some probability space) taking values in $[0,1]$
satisfying 
\begin{itemize}
   \item $U$ obeys the Lebesgue measure.
   \item $V$ satisfies $\mathbb{E}|U-V| \leq \varepsilon$.
\end{itemize}                    
The convexity/concavity properties of mutual information, specifically 
Lemma~\ref{lemma: concavity/convexity of mutual information}.(2), imply that $r(\varepsilon)$ is a convex function in $\varepsilon$
(c.f.~\cite[Lemma 10.4.1]{Cover--Thomas}.)
Thus it follows from (\ref{eq: lower bound on mutual information in the case of Hilbert cube}) and 
(\ref{eq: distortion in the case of Hilbert cube}) that 
\[ \frac{I(X;Y)}{n} \geq \frac{1}{n} \sum_{m=0}^{n-1} r\left(\mathbb{E}|X_m-Y_{m,0}|\right) 
    \geq r\left(\frac{1}{n} \sum_{m=0}^{n-1} \mathbb{E}|X_m-Y_{m,0}|\right)  \geq r(\varepsilon), \]
and hence $R_\mu(\varepsilon) \geq r(\varepsilon)$.
Then $R_\mu(\varepsilon)\sim |\log \varepsilon|$ follows from the next claim.

\begin{claim}  \label{claim: rate distortion function of iid on [0,1]}
\[  r(\varepsilon)\sim |\log \varepsilon| \quad (\varepsilon\to 0). \]
\end{claim}
\begin{proof}
It is again easy to prove $\limsup_{\varepsilon\to 0} r(\varepsilon)/|\log \varepsilon| \leq 1$.
So we prove a lower bound on $r(\varepsilon)$.
Let $U$ and $V$ be random variables in the above definition of $r(\varepsilon)$.
Fix $D>1$ and set $l=\lfloor 1/(D\varepsilon)\rfloor$.
We define a partition $\mathcal{P}$ of $[0,1]$ by 
\[ \mathcal{P} = \left\{[0,D\varepsilon), [D\varepsilon, 2D\varepsilon), [2D\varepsilon, 3D\varepsilon), \dots, 
                         [lD\varepsilon, 1] \right\}. \]
For $u\in[0,1]$ we denote by $\mathcal{P}(u)$ the atom of $\mathcal{P}$ containing $u$.
It follows from $\mathbb{E}|U-V| \leq \varepsilon$ that 
\[ \mathbb{P}\left(|U-V|\geq D\varepsilon\right) \leq \frac{\mathbb{E}|U-V|}{D\varepsilon}  \leq \frac{1}{D}. \]
By the data-processing inequality 
\[ I(U;V) \geq I\left(\mathcal{P}(U);V\right) = H\left(\mathcal{P}(U)\right) - H\left(\mathcal{P}(U)|V\right). \]
Under the condition $|U-V| < D\varepsilon$, if we know $V$ then the number of possibilities of $\mathcal{P}(U)$ is at most three.
This implies 
\[ H(\mathcal{P}(U)|V) \leq \log 3 +  \mathbb{P}\left(|U-V|\geq D\varepsilon\right) \log (l+1) \leq 
    \log 3 + \frac{\log(l+1)}{D}. \]
Since $U$ obeys the Lebesgue measure, $H\left(\mathcal{P}(U)\right)$ is bounded from below by 
\[  l (D\varepsilon) \log (1/D\varepsilon)  \geq    (1-D\varepsilon) \log (1/D\varepsilon). \]
Thus 
\[ r(\varepsilon)  \geq       (1-D\varepsilon) \log (1/D\varepsilon) - \frac{\log \left(1+\lfloor 1/(D\varepsilon)\rfloor\right)}{D} - \log 3. \]
It follows 
\[  \liminf_{\varepsilon\to 0}\frac{r(\varepsilon)}{|\log\varepsilon|} \geq 1-\frac{1}{D}. \]
Letting $D\to \infty$ we get $\liminf_{\varepsilon\to 0} r(\varepsilon)/|\log \varepsilon| \geq 1$.
\end{proof}
\end{example}

\section{Proof of Theorem \ref{thm: main theorem}} \label{section: proof of main theorem}

In this section we prove Theorem \ref{thm: main theorem}.
Throughout this section $(\mathcal{X},T)$ is a dynamical system, and $d$ a metric on $\mathcal{X}$.
Recall that for $n \geq 1$ we defined the distance $d_n$ on $\mathcal{X}$ by 
\[ d_n(x,y) = \max_{0\leq k<n} d(T^k x, T^k y). \]
We define another distance $\bar{d}_n$ on $\mathcal{X}$ by 
\[ \bar{d}_n(x,y) = \frac{1}{n} \sum_{k=0}^{n-1} d(T^k x, T^k y). \]
Obviously $\bar{d}_n(x,y)\leq d_n(x,y)$.
For $\varepsilon>0$ we set 
\[ \tilde{S}(\mathcal{X},T,d, \varepsilon) = \lim_{n\to \infty} \frac{1}{n} \log \#(\mathcal{X},\bar{d}_n,\varepsilon). \]
This limit exists because $\log \#(\mathcal{X},\bar{d}_n,\varepsilon)$ is a subaddtive function of $n$.
We have 
\begin{equation}  \label{eq: S dominates S bar}
    \tilde{S}(\mathcal{X},T,d,\varepsilon) \leq S(\mathcal{X},T,d, \varepsilon) = 
     \lim_{n\to \infty} \frac{1}{n} \log \#(\mathcal{X},d_n,\varepsilon). 
\end{equation}

\subsection{Metric mean dimension dominates rate distortion functions}  
\label{subsection: metric mean dimension dominates rate distortion functions}

\begin{lemma} \label{lemma: metric mean dimension dominates rate distortion function}
For $\varepsilon>0$ and every invariant probability measure $\mu$ on $\mathcal{X}$ we have 
\[  R_\mu(\varepsilon) \leq \tilde{S}(\mathcal{X},T,d,\varepsilon) \leq S(\mathcal{X},T,d,\varepsilon). \]
\end{lemma}

\begin{proof}
Let $n>0$,
and let $\{U_1,\dots, U_K\}$ be an open covering of $\mathcal{X}$ such that every $U_k$ has diameter smaller than $\varepsilon$
with respect to the distance $\bar{d}_n$.
We choose a point $p_k\in U_k$ for each $k$.
We define a map $f:\mathcal{X}\to \{p_1,\dots,p_K\}$ by setting 
$f(x)=p_k$ where $k$ is the smallest number satisfying $x\in U_k$.
Obviously $\bar{d}_n(x, f(x)) <\varepsilon$.
Let $X$ be a random variable obeying $\mu$.
We set $Y = (f(X), Tf(X),\dots, T^{n-1}f(X))$.
This satisfies the distortion condition (\ref{eq: distortion condition}):
\[  \mathbb{E}\left(\frac{1}{n}\sum_{k =0}^{n-1}d(T^k X, T^k f(X))\right) 
    = \mathbb{E} \bar{d}_n(X, f(X)) < \varepsilon. \]
The mutual information $I(X;Y)$ is bounded by 
\[ I(X;Y) \leq H(Y) \leq \log K, \] 
where the second inequality holds because $Y$ takes at most $K$ values.
This shows $R_\mu(\varepsilon) \leq \tilde{S}(\mathcal{X},T,d,\varepsilon)$.
\end{proof}

Lemma \ref{lemma: metric mean dimension dominates rate distortion function} immediately implies 
one direction of Theorem \ref{thm: main theorem}:
\begin{equation}  \label{eq: easier part of variational principle}
  \overline{\mdim}_\mathrm{M}(\mathcal{X},T,d) \geq 
     \limsup_{\varepsilon\to 0} \frac{\sup_{\mu\in \mathscr{M}^T(\mathcal{X})}R_\mu(\varepsilon)}{|\log\varepsilon|}. 
\end{equation}     
The case of $\underline{\mdim}_\mathrm{M}(\mathcal{X},T,d)$ is the same.
Notice that we have not used Condition \ref{condition: distance} so far.

\subsection{Condition \ref{condition: distance} implies that $d_n$ and $\bar{d}_n$ look the same}
\label{subsection: implication of Condition 1.1}

This subsection is the only place where Condition \ref{condition: distance} plays a role.
We set $[n] = \{0,1,2,\dots,n-1\}$.
For a finite subset $A\subset \mathbb{Z}$ we define $d_A(x,y) = \max_{a\in A} d(T^a x, T^a y)$ for $x,y\in \mathcal{X}$.
In particular $d_n = d_{[n]}$.

\begin{lemma} \label{lemma: comparison between d_N and d^bar_N}
For any natural number $n$ and any real numbers $\varepsilon>0$ and $L>1$ we have 
\[ \frac{1}{n}\log\#(\mathcal{X},d_n,2L\varepsilon) \leq \log 2 + \frac{1}{L}\log\#(\mathcal{X},d,\varepsilon) 
    + \frac{1}{n} \log\#(\mathcal{X}, \bar{d}_n,\varepsilon). \]
\end{lemma}

\begin{proof}
Let $X= W_1\cup\dots\cup W_M$ be an open covering such that $\diam(W_m, d)<\varepsilon$ for all $1\leq m\leq M$ and 
$M=\#(\mathcal{X},d,\varepsilon)$.
We also take an open covering $X=U_1\cup\dots\cup U_N$ such that
$\diam(U_i, \bar{d}_n)<\varepsilon$ for all $1\leq i\leq N$ and $N=\#(\mathcal{X},\bar{d}_n,\varepsilon)$.

We choose a point $p_i\in U_i$ for each $1\leq i \leq N$.
Every point $x\in U_i$ satisfies
$\bar{d}_n(x,p_i)<\varepsilon$,
and hence 
\[ |\{0 \leq k\leq n-1 |\, d(T^k x, T^k p_i)\geq L\varepsilon\}| < \frac{n}{L}. \]
It follows that $U_i$ is contained in the union of the open balls 
\[  B_{L\varepsilon}(p_i, d_{[n]\setminus A}) , \]
where $A$ runs over subsets of $[n] = \{0,1,2,\dots,n-1\}$ satisfying $|A|< n/L$. 
For $A= \{k_1, \dots, k_a\} \subset [n]$ with $a<n/L$, the ball $B_{L\varepsilon}(p_i, d_{[n]\setminus A})$ is equal to the union of 
\begin{equation} \label{eq: ball cap W}
 B_{L\varepsilon}(p_i, d_{[n]\setminus A}) \cap T^{-k_1}W_{m_1}\cap \dots \cap T^{-k_a}W_{m_a}, \quad 
   (1\leq m_1,\dots, m_a \leq M).
\end{equation} 
The sets (\ref{eq: ball cap W}) have diameter less than $2L\varepsilon$ with respect to the distance 
$d_n$. 
Hence 
\[ \#(B_{L\varepsilon}(p_i, d_{[n]\setminus A}), d_n, 2L\varepsilon) \leq M^a \leq M^{n/L}. \]
There are $N$ choices of $U_i$ and $2^n$ choices of $A\subset [n]$. Thus
\[ \#(X, d_n, 2L\varepsilon) \leq 2^n M^{n/L} N. \]
This proves the statement.
\end{proof}

\begin{lemma}   \label{lemma: implication of metric condition}
Under Condition \ref{condition: distance}, 
\begin{align*}
 \overline{\mdim}_\mathrm{M}(\mathcal{X},T,d) & = \limsup_{\varepsilon\to 0} \frac{\tilde{S}(\mathcal{X},T,d,\varepsilon)}{|\log\varepsilon|},\\
   \underline{\mdim}_\mathrm{M}(\mathcal{X},T,d) 
   & = \liminf_{\varepsilon\to 0} \frac{\tilde{S}(\mathcal{X},T,d,\varepsilon)}{|\log\varepsilon|}.  
\end{align*}
\end{lemma}

\begin{proof}
We prove the equality for $\overline{\mdim}_\mathrm{M}(\mathcal{X},T,d)$.
The case of $\underline{\mdim}_\mathrm{M}(\mathcal{X},T,d)$ is the same.
From $S(\mathcal{X},T,d,\varepsilon)  \geq \tilde{S}(\mathcal{X},T,d,\varepsilon)$, the inequality 
\[   \overline{\mdim}_\mathrm{M}(\mathcal{X},T,d)   = \limsup_{\varepsilon\to 0}
    \frac{S(\mathcal{X},T,d,\varepsilon)}{|\log \varepsilon|} 
   \geq  \limsup_{\varepsilon\to 0} \frac{\tilde{S}(\mathcal{X},T,d,\varepsilon)}{|\log\varepsilon|} \]
is obvious.
Take $0< \delta <1$ and apply Lemma \ref{lemma: comparison between d_N and d^bar_N} with $L=(1/\varepsilon)^\delta$.
Then we get 
\[ \frac{1}{n} \log \#(\mathcal{X},d_n,2\varepsilon^{1-\delta}) \leq 
   \log 2 + \frac{\log \#(\mathcal{X},d,\varepsilon)}{(1/\varepsilon)^\delta} + 
    \frac{1}{n} \log \#(\mathcal{X},\bar{d}_n,\varepsilon).  \]
Letting $n\to \infty$
\[  S(\mathcal{X},T,d, 2\varepsilon^{1-\delta}) \leq \log 2 + \varepsilon^\delta \log\#(\mathcal{X},d,\varepsilon)
     +  \tilde{S}(\mathcal{X},T,d,\varepsilon). \]
By Condition \ref{condition: distance}, the second term in the right-hand side goes to zero as $\varepsilon\to 0$ (this is the only place where we use Condition \ref{condition: distance}).
It follows that 
\[   (1-\delta)\cdot \overline{\mdim}_\mathrm{M}(\mathcal{X},T,d) 
       \leq \limsup_{\varepsilon\to 0} \frac{\tilde{S}(\mathcal{X},T,d,\varepsilon)}{|\log\varepsilon|} . \]
Letting $\delta\to 0$, we get the statement.
\end{proof}

\subsection{Completion of the proof of Theorem \ref{thm: main theorem}}

For $n\geq 1$ we define a distance $\bar{d}_n$ on $\mathcal{X}^n$ by 
\[ \bar{d}_n\left((x_0,\dots, x_{n-1}), (y_0,\dots, y_{n-1})\right) 
   = \frac{1}{n} \sum_{k=0}^{n-1} d(x_k, y_k). \]
In particular 
\[ \bar{d}_n(x,y) = \bar{d}_n\left((x,T x, \dots, T^{n-1}x), (y, Ty, \dots, T^{n-1}y)\right) \quad (x,y\in \mathcal{X}). \]

\begin{proposition}  \label{prop: construction of invariant measure}
For any real numbers $\varepsilon>0$ and $D>2$
there exists an invariant probability measure $\mu$ on $\mathcal{X}$ satisfying 
\[ R_\mu(\varepsilon) \geq \left(1-\frac{1}{D}\right)\tilde{S}(\mathcal{X},T,d,(12D+4)\varepsilon). \]
\end{proposition}

\begin{proof}
For each $n \geq 1$ we choose $S_n \subset \mathcal{X}$ a maximal $(6D+2)\varepsilon$-separated set
with respect to the distance $\bar{d}_n$.  
It follows 
\begin{equation} \label{eq: separated set}
  |S_n| \geq \#(\mathcal{X},\bar{d}_n, (12D+4)\varepsilon).
\end{equation}
Let $\nu_n$ be the uniform distribution over $S_n$: 
\[ \nu_n = \frac{1}{|S_n|} \sum_{p\in S_n} \delta_p. \]
Set 
\[ \mu_n = \frac{1}{n} \sum_{k=0}^{n-1} T^k_* \nu_n. \]
We can choose a subsequence $\{\mu_{n_i}\}_{i=1}^\infty$ converging to an invariant probability 
measure $\mu$ in the weak$^*$ topology.
We prove that this $\mu$ satisfies the statement.

We choose a partition $\mathcal{P} = \{P_1,\dots, P_K\}$ of $\mathcal{X}$ such that 
\begin{itemize}
   \item Every $P_k$ has diameter smaller than $\varepsilon$ with respect to the distance $d$.
   \item $\mu(\partial P_k)= 0$ for all $1\leq k \leq K$.
\end{itemize}
We choose a point $p_k\in P_k$ for each $1\leq k\leq K$.
Set $A=\{p_1,\dots,p_K\}$.
We define a map $\mathcal{P}:\mathcal{X}\to A$ by 
$\mathcal{P}(x) = p_k$ for $x\in P_k$.
It follows that
\begin{equation} \label{eq: diameter of partition}
   d(x, \mathcal{P}(x)) < \varepsilon.
\end{equation}
For $n\geq 1$ we set 
\[ \mathcal{P}^n(x) = (\mathcal{P}(x), \mathcal{P}(T x), \dots, \mathcal{P}(T^{n-1} x)). \]

\begin{claim}   \label{claim: P_0^{N-1}(S_N)}
   \begin{enumerate}
     \item  The set 
       \[ \mathcal{P}^n(S_n) = \{\mathcal{P}^n(x)|\, x\in S_n\} \]
       is a $6D\varepsilon$-separated set with respect to the distance $\bar{d}_n$.
     \item The push-forward measure $\mathcal{P}^n_*\nu_n$ is the uniform distribution over
     $\mathcal{P}^n(S_n)$. Moreover $|\mathcal{P}^n(S_n)| = |S_n|$.
 \end{enumerate} 
\end{claim}

\begin{proof}
By (\ref{eq: diameter of partition}) we have $\bar{d}_n\left((x, Tx, \dots, T^{n-1}x), \mathcal{P}^n (x)\right) < \varepsilon$.
For any two distinct points $x,y$ in $S_n$, the distance $\bar{d}_n(\mathcal{P}^n(x), \mathcal{P}^n(y))$
is bounded from below by 
\begin{equation*}
  \begin{split}
    & \bar{d}_n(x,y) -  \bar{d}_n\left((x, Tx, \dots, T^{n-1}x), \mathcal{P}^n(x)\right) 
       - \bar{d}_n \left((y, Ty, \dots, T^{n-1} y), \mathcal{P}^n(y)\right) \\
    & \geq (6D+2) \varepsilon -2\varepsilon = 6D\varepsilon. 
  \end{split} 
\end{equation*}   
This proves part (1) of the claim. Moreover it shows that the map 
\[  S_n \ni x\mapsto \mathcal{P}^n(x) \in \mathcal{P}^n(S_n) \]
is bijective. Since $\nu_n$ is uniformly distributed over $S_n$, the measure $\mathcal{P}^n_*\nu_n$ 
is uniformly distributed over $\mathcal{P}^n(S_n)$.
This establishes part (2).
\end{proof}

Consider random variables $X$ and $Y = (Y_0,\dots,Y_{m-1})$
defined on a probability space $(\Omega, \mathbb{P})$ such that 
$\mathrm{Law}(X) = \mu$ and $Y_i$ take values in $\mathcal{X}$ with
\begin{equation} \label{eq: distortion condition in the proof}
  \mathbb{E}\left(\frac{1}{m} \sum_{i=0}^{m-1} d(T^i X, Y_i) \right) < \varepsilon. 
\end{equation}  
We estimate the mutual information $I(X;Y)$ from below.
As in Remark \ref{remark: reduction to discrete case}, we can assume that the distribution of $Y$ is
supported on a finite set $\mathcal{Y}\subset \mathcal{X}^m$.
From the data-processing inequality (Lemma \ref{lemma: data-processing inequality})
\[  I(X;Y) \geq I(\mathcal{P}^m(X); Y). \]
So it is enough to estimate $I(\mathcal{P}^m(X); Y)$ from below.
Let $\tau =\mathrm{Law}(\mathcal{P}^m(X),Y)$ be the law of $(\mathcal{P}^m (X),Y)$, 
which is a probability measure on 
$A^m \times \mathcal{Y}$.
It follows that
\begin{equation}  \label{eq: distortion condition in the proof: integral form}
   \begin{split}
    \int_{A^m \times \mathcal{Y}} \bar{d}_m(x, y)\,  d\tau(x,y)  &= 
    \mathbb{E}\left(\frac{1}{m}\sum_{i=0}^{m-1} d(\mathcal{P}(T^i X), Y_i)\right) \\
     &\leq \varepsilon +  \mathbb{E}\left(\frac{1}{m} \sum_{i=0}^{m-1} d(T^i X, Y_i) \right) < 2\varepsilon. 
   \end{split}
\end{equation}   
Here we used $d(\mathcal{P}(T^i X), T^i X) < \varepsilon$ and (\ref{eq: distortion condition in the proof}).

For each $n \geq 1$ we choose a probability measure $\pi_n$ on $A^m \times A^m$ such that 
\begin{itemize}
  \item $\pi_n$ is a coupling of $(\mathcal{P}^m_*\mu_n, \mathcal{P}^m_*\mu)$, 
  namely its first and second marginals are $\mathcal{P}^m_*\mu_n$ and $\mathcal{P}^m_*\mu$ respectively.
  \item $\pi_n$ minimizes the integral
  \[  \int_{A^m\times A^m} \bar{d}_m(x,y) d\pi(x,y) \]
  among all couplings $\pi$ of $(\mathcal{P}^m_*\mu_n, \mathcal{P}^m_*\mu)$.
\end{itemize}
(These two conditions means that $\pi_n$ is an optimal transference plan in the language of 
Optimal Transport.)

\begin{claim}  \label{claim: weak convergence of pi_N}
The sequence $\pi_{n_i}$ converges to $(\mathcal{P}^m\times \mathcal{P}^m)_*\mu$ in the weak$^*$ topology.
\end{claim}

\begin{proof}
Since $\mu(\partial P_k)=0$, the sequence $\mathcal{P}^m_*\mu_{n_i}$ converges to $\mathcal{P}^m_*\mu$.
Then the statement becomes a very special case of a theorem of optimal transport
\cite[Theorem 5.20]{Villani}.
As all the measures here are supported on finite sets, our
situation is simpler than the general setting in \cite{Villani}, and
we provide a self-contained elementary proof in Lemma \ref{lemma: convergence of optimal transport plan} in the Appendix.
\end{proof}

Both the second marginal of $\pi_n$ and the first marginal of $\tau$ are equal to the measure $\mathcal{P}^m_*\mu$.
So we can compose them and produce a coupling $\tau_n$ of
$(\mathcal{P}^m_*\mu_n, \mathrm{Law}(Y))$.
Namely 
\[ \tau_n(x, y) = \sum_{x'\in A^m} \pi_n(x, x') 
     \mathbb{P}(Y=y|\mathcal{P}^m(X)=x'), \quad 
      (x\in A^m, y\in \mathcal{Y}). \]
Here we identify probability measures with their probability mass functions.
From Claim \ref{claim: weak convergence of pi_N} the measures $\tau_{n_i}$ converge to $\tau$ in the weak$^*$ topology.
In particular, it follows from (\ref{eq: distortion condition in the proof: integral form}) that
\begin{equation} \label{eq: distortion of tau_N}
    \mathbb{E}_{\tau_{n_i}} (\bar{d}_m(x,y)) := 
    \int_{A^m \times \mathcal{Y}} \bar{d}_m(x,y) \, d\tau_{n_i}(x,y) < 2\varepsilon 
\end{equation}    
for all sufficiently large $n_i$.

We define a conditional probability mass function $\tau_n(y|x)$ by 
\[ \tau_n(y| x) = \frac{\tau_n(x,y)}{\mathcal{P}^m_*\mu_n(x)}. \]
This is defined for 
\[ x \in \bigcup_{k=0}^{n-1} \mathcal{P}^m\left(T^{k} S_n\right), \quad y\in \mathcal{X}^{m}.  \]

Take $n\geq 2m$ and 
let $n=qm + r$ with $m\leq r\leq 2m-1$. Fix a point $a\in \mathcal{X}$.
We denote by $\delta_a(\cdot)$ the delta probability measure at $a$ on $\mathcal{X}$.
For $x =(x_1,\dots,x_n) \in \mathcal{P}^n(S_n)$ we let $x_k^l$ denote the $(l-k+1)$-tuple $x_k^l=  (x_k,\dots, x_l)$ for $0\leq k\leq l<n$. For such an $x$ we
define probability mass functions 
$\sigma_{n,0}(\cdot|x),\dots, \sigma_{n,m-1}(\cdot|x)$ on $\mathcal{X}^n$ as follows:
\begin{equation}  \label{eq: definition of conditional probability functions}
   \begin{split}
    \sigma_{n,0}(y|x) =& \prod_{j=0}^{q-1} \tau_n\left(y_{jm}^{jm+m-1}|x_{jm}^{jm+m-1}\right)
                                                 \cdot \prod_{k=n-r}^{n-1} \delta_a(y_k), \\
   \sigma_{n,1}(y|x)  =&  \delta_a(y_0)\cdot \prod_{j=0}^{q-1}\tau_n \left(y_{jm+1}^{jm+m}| x_{jm+1}^{jm+m}\right) \cdot 
                                                   \prod_{k=n-r+1}^{n-1} \delta_a(y_k), \\
                                                   &\dots \\
   \sigma_{n,m-1}(y|x) = & \prod_{k=0}^{m-2}\delta_a(y_k) \cdot \prod_{j=0}^{q-1} \tau_n \left(y_{jm+m-1}^{jm+2m-2}| x_{jm+m-1}^{jm+2m-2}\right) 
                                    \cdot \prod_{k=n-r+m-1}^{n-1}\delta_a(y_k).                
   \end{split}
\end{equation}
See Figure \ref{figure: definition of sigma_{n,i}}.
Finally we set 
\[  \sigma_n(y|x) = \frac{\sigma_{n,0}(y|x) + \sigma_{n,1}(y|x) + \dots + \sigma_{n,m-1}(y|x)}{m}. \]

\begin{figure}[!htbp]
	\centering
    \includegraphics[width=\textwidth,bb= 0 0 480 80]{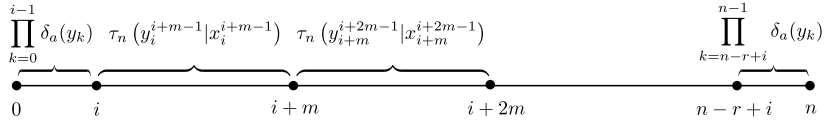}
    \caption{Definition of $\sigma_{n,i}(y|x)$}
	\label{figure: definition of sigma_{n,i}}
\end{figure}

\begin{claim}  \label{claim: lower bound on mutual information}
\[  \frac{1}{m} I\left(\mathcal{P}^m_*\mu_n, \tau_n\right) 
     \geq  \frac{1}{n} I\left(\mathcal{P}^n_*\nu_n, \sigma_n\right). \] 
     Here $I\left(\mathcal{P}^m_*\mu_n, \tau_n\right)$ and $I\left(\mathcal{P}^n_*\nu_n, \sigma_n\right)$
     are the mutual informations of the probability distributions 
\[ \mathcal{P}^m_*\mu_n(x) \tau_n(y|x), \quad \mathcal{P}^n_*\nu_n(x)\sigma_n(y|x) \]
respectively.
\end{claim}

\begin{proof}
We use the concavity/convexity of mutual information (Lemma \ref{lemma: concavity/convexity of mutual information}).
From the convexity 
\[  I\left(\mathcal{P}^n_*\nu_n, \sigma_n\right) \leq \frac{1}{m} \sum_{i=0}^{m-1}
    I\left(\mathcal{P}^n_*\nu_n, \sigma_{n,i} \right).    \]
From (\ref{eq: definition of conditional probability functions}) and the subadditivity of mutual information 
(Lemma \ref{lemma: subadditivity of mutual information})
\[  I\left(\mathcal{P}^n_*\nu_n, \sigma_{n,i} \right) \leq 
    \sum_{j=0}^{q-1}  I\left(\mathcal{P}^m_*(T^{i+jm}_*\nu_n), \tau_n\right)  . \]
Therefore     
\begin{equation*}
  \begin{split}
     I\left(\mathcal{P}^n_*\nu_n, \sigma_n\right)  &\leq \frac{1}{m}  \sum_{i=0}^{m-1}  \sum_{j=0}^{q-1}
      I\left(\mathcal{P}^m_*(T^{i+jm}_*\nu_n), \tau_n\right) \\
      &= \frac{1}{m}\sum_{k=0}^{qm-1} I\left(\mathcal{P}^m_*(T^k_*\nu_n), \tau_n\right) \\
      &\leq     \frac{1}{m} \sum_{k=0}^{n-1}  I\left(\mathcal{P}^m_*(T^k_*\nu_n), \tau_n\right) \\
     &\leq  \frac{n}{m} I\left(\mathcal{P}^m_*\left(\frac{1}{n}\sum_{k=0}^{n-1}T^k_*\nu_n \right), \tau_n\right)\\
     &\qquad\qquad (\text{by the concavity in Lemma \ref{lemma: concavity/convexity of mutual information} (1)}) \\
      & = \frac{n}{m} I\left(\mathcal{P}^m_*\mu_n, \tau_n\right) \quad 
      (\text{by }\mu_n = \frac{1}{n} \sum_{k=0}^{n-1}T^k_*\nu_n).      
  \end{split}
\end{equation*}  
We would like to remark that the above calculation is quite analogous to Misiurewicz's proof
\cite{Misiurewicz} of the standard variational principle.
\end{proof}

\begin{claim}  \label{claim: distortion is sufficiently small}
We denote by $\mathbb{E}_{\mathcal{P}^n_*\nu_n, \sigma_n}(\bar{d}_n(x,y))$ the expected value of $\bar{d}_n(x,y)$ 
 $(x,y\in \mathcal{X}^n)$
with respect to the probability measure 
\[  \mathcal{P}^n_*\nu_{n}(x) \sigma_{n}(y|x).  \]
Then $\mathbb{E}_{\mathcal{P}^{n_i}_*\nu_{n_i}, \sigma_{n_i}} (\bar{d}_{n_i}(x,y)) <3\varepsilon$ for sufficiently large $n_i$.
Moreover 
\begin{equation}  \label{eq: mutual information is sufficiently large}
    I\left(\mathcal{P}^{n_i}_*\nu_{n_i}, \sigma_{n_i}\right) \geq 
    \left(1-\frac{1}{D}\right) \log |S_{n_i}| - H(1/D) \quad 
    \text{for sufficiently large $n_i$}.
\end{equation}    
\end{claim}

\begin{proof}
\[ \mathbb{E}_{\mathcal{P}^n_*\nu_n, \sigma_n} \left(\bar{d}_n(x,y)\right) = 
   \frac{1}{m} \sum_{i=0}^{m-1} \mathbb{E}_{\mathcal{P}^n_*\nu_n, \, \sigma_{n,i}}\left(\bar{d}_n(x,y)\right).    \]
By (\ref{eq: definition of conditional probability functions})
\begin{equation*}
   \frac{1}{m}    \mathbb{E}_{\mathcal{P}^n_*\nu_n, \, \sigma_{n,i}}\left(\bar{d}_n(x,y)\right)    
    \leq \frac{1}{n} \sum_{j=0}^{q-1} \mathbb{E}_{\mathcal{P}^m_*(T^{i+jm})_*\nu_n, \tau_n}\left(\bar{d}_m(x',y')\right)
    + \frac{r \cdot \diam(\mathcal{X},d)}{mn}.
\end{equation*}
Here $x,y$ are random points in $\mathcal{X}^n$, whereas $x',y'$ are in $\mathcal{X}^m$.
Therefore 
\begin{equation*}
   \begin{split}
   \mathbb{E}_{\mathcal{P}^n_*\nu_n, \sigma_n} \left(\bar{d}_n(x,y)\right)  &\leq 
     \frac{1}{n} \sum_{i=0}^{m-1} \sum_{j=0}^{q-1}  \mathbb{E}_{\mathcal{P}^m_*(T^{i+jm})_*\nu_n, \, \tau_n}\left(\bar{d}_m(x',y')\right)
     + \frac{r\cdot \diam(\mathcal{X},d)}{n} \\
     &= \frac{1}{n} \sum_{k=0}^{qm-1} \mathbb{E}_{\mathcal{P}^m_*(T^k_*\nu_n),\, \tau_n}\left(\bar{d}_m(x',y')\right)
     + \frac{r\cdot \diam(\mathcal{X},d)}{n} \\
     & \leq  \frac{1}{n} \sum_{k=0}^{n-1} \mathbb{E}_{\mathcal{P}^m_*(T^k_*\nu_n),\, \tau_n}\left(\bar{d}_m(x',y')\right)
     + \frac{r\cdot \diam(\mathcal{X},d)}{n}  \\
     & = \mathbb{E}_{\mathcal{P}^m_*\mu_n,\, \tau_n}\left(\bar{d}_m(x',y')\right) 
     + \frac{r\cdot \diam(\mathcal{X},d)}{n}.
   \end{split}
\end{equation*}   
In the last line we used $\mu_n = (1/n)\sum_{k=0}^{n-1} T^k_*\nu_n$.
As a conclusion, 
\[ \mathbb{E}_{\mathcal{P}^n_*\nu_n, \sigma_n} \left(\bar{d}_n(x,y)\right)  \leq 
   \int_{A^m\times \mathcal{Y}} \bar{d}_m(x',y') d\tau_n(x',y') + \frac{r\cdot \diam(\mathcal{X},d)}{n}. \]
By (\ref{eq: distortion of tau_N}) and $r\leq 2m-1$, this is bounded by $3\varepsilon$ for sufficiently large 
$n=n_i$.

By Claim \ref{claim: P_0^{N-1}(S_N)}, $\mathcal{P}^n_* \nu_n$ is uniformly distributed over $\mathcal{P}^n(S_n)$, which is a
$(6D\varepsilon)$-separated set of cardinarity $|S_n|$.
Then (\ref{eq: mutual information is sufficiently large}) follows from
Corollary \ref{cor: uniform distribution over separated set}.
\end{proof}

We conclude that for sufficiently large $n_i$
\begin{equation*}
   \begin{split}
     \frac{1}{m} I\left(\mathcal{P}^m_*\mu_{n_i}, \tau_{n_i}\right)  &\geq  \frac{1}{n_i} I(\mathcal{P}^{n_i}_*\nu_{n_i},\sigma_{n_i})
          \quad (\text{by Claim \ref{claim: lower bound on mutual information}})   \\
     &\geq  \left(1-\frac{1}{D}\right) \frac{\log |S_{n_i}|}{n_i} - \frac{H(1/D)}{n_i}  
          \quad (\text{by Claim \ref{claim: distortion is sufficiently small}}) \\
     &\geq  \left(1-\frac{1}{D}\right) \frac{\log \#(\mathcal{X},\bar{d}_{n_i}, (12D+4)\varepsilon)}{n_i} - \frac{H(1/D)}{n_i}\\
     &\qquad\qquad  (\text{by (\ref{eq: separated set})}).
    \end{split}
\end{equation*}
   
The probability measures $\tau_{n_i}(x,y)$ converge to $\tau = \mathrm{Law}(\mathcal{P}^m(X), Y)$ in the weak$^*$ topology.
Therefore it follows from Lemma \ref{lemma: convergence of mutual information} that
\[ \frac{1}{m} I(\mathcal{P}^m(X); Y) \geq  \left(1-\frac{1}{D}\right) \tilde{S}(\mathcal{X},T,d, (12D+4)\varepsilon). \]
From the data-processing inequality (Lemma \ref{lemma: data-processing inequality}) 
\[  \frac{1}{m} I(X; Y) \geq  \left(1-\frac{1}{D}\right) \tilde{S}(\mathcal{X},T,d, (12D+4)\varepsilon). \]
This proves the statement.
\end{proof}

Lemma \ref{lemma: metric mean dimension dominates rate distortion function} and 
Proposition \ref{prop: construction of invariant measure} immediately imply:

\begin{corollary}  \label{cor: one version of variational principle}
\begin{equation*}
   \begin{split}
      \limsup_{\varepsilon\to 0} \frac{\tilde{S}(\mathcal{X},T,d,\varepsilon)}{|\log\varepsilon|}
     &= \limsup_{\varepsilon\to 0} \frac{\sup_{\mu\in \mathscr{M}^T(\mathcal{X})}R_\mu(\varepsilon)}{|\log \varepsilon|} \\
      \liminf_{\varepsilon\to 0} \frac{\tilde{S}(\mathcal{X},T,d,\varepsilon)}{|\log\varepsilon|}
     &= \liminf_{\varepsilon\to 0} \frac{\sup_{\mu\in \mathscr{M}^T(\mathcal{X})}R_\mu(\varepsilon)}{|\log \varepsilon|}.
   \end{split}  
\end{equation*}     
\end{corollary}

Theorem \ref{thm: main theorem} follows from 
Lemma \ref{lemma: implication of metric condition} and Corollary \ref{cor: one version of variational principle}.

\section{Proof of Theorem \ref{thm: second main theorem}} \label{section: proof of second main theorem}

Here we prove Theorem \ref{thm: second main theorem}.
The proof is very close to that of Theorem \ref{thm: main theorem}, and in view of this our
 explanation is more concise.
Throughout this section, $(\mathcal{X},T)$ is a dynamical system with a distance $d$.
For $x=(x_0,\dots,x_{n-1})$ and $y=(y_0,\dots,y_{n-1})$ in $\mathcal{X}^n$ we set 
\[ d_n(x,y) = \max_{0\leq i\leq n-1} d(x_i,y_i). \]

\begin{lemma}  \label{lemma: upper bound, second main theorem}
For every $\varepsilon>0$ and every invariant probability measure $\mu$ on $\mathcal{X}$ we have 
\[  \tilde{R}_\mu(\varepsilon) \leq S(\mathcal{X},T,d,\varepsilon). \]
\end{lemma}

\begin{proof}
Let $n>0$ and choose an open covering $\{U_1,\dots,U_K\}$ of $\mathcal{X}$ such that every $U_k$ has diameter less than $\varepsilon$
with respect to $d_n$. Choose a point $p_k\in U_k$ for each $k$.
We define $f:\mathcal{X}\to \{p_1,\dots, p_K\}$ by $f(x)=p_k$ where $k$ is the smallest integer satisfying $x\in U_k$.
Then $d_n(x,f(x)) < \varepsilon$.
Let $X$ be a random variable obeying $\mu$, and set 
$Y= (f(X), Tf(X), \dots, T^{n-1}f(X))$.
We have $d_n(X, f(X))<\varepsilon$ almost surely.
It follows that
\[ \mathbb{E}\left(\text{the number of $i\in [0,n-1]$ with $d(T^i X, T^i f(X))\geq \varepsilon$}\right) = 0. \]
Thus $(X, Y)$ satisfies the distortion condition (\ref{eq: modified distortion condition}) for any $\alpha>0$.
Since $Y$ takes at most $K$ values
\[ I(X;Y) \leq H(Y) \leq \log K. \]
This proves the statement.
\end{proof}

\begin{proposition} \label{prop: lower bound, second main theorem}
For any positive number $\varepsilon$ there exists an invariant probability measure $\mu$ on $\mathcal{X}$ satisfying 
\[ \tilde{R}_\mu(\varepsilon) \geq S(\mathcal{X},T,d, 12\varepsilon). \]
\end{proposition}

\begin{proof}
For each $n\geq 1$ we take a maximal $6\varepsilon$-separated set $S_n\subset \mathcal{X}$ with respect to 
the distance $d_n$. It follows $|S_n|\geq \#(\mathcal{X}, d_n, 12\varepsilon)$.
Let $\nu_n$ be the uniform distribution over $S_n$ and set 
\[   \mu_n = \frac{1}{n} \sum_{k=0}^{n-1} T^k_*\nu_n. \]
Choose a subsequence $\{n_i\}$ so that $\mu_{n_i}$ converges to $\mu\in \mathscr{M}^T(\mathcal{X})$ in the weak$^*$ topology.
We prove that this $\mu$ satisfies the statement.
For $n\geq 1$, $x=(x_0,\dots,x_{n-1})$ and $y=(y_0,\dots,y_{n-1})$ in $\mathcal{X}^n$ we set 
\[ f_n(x,y) = \text{the number of $k\in [0,n-1]$ satisfying $d(x_k,y_k)\geq 2\varepsilon$}. \]
Here we chose ``$2\varepsilon$'' for the later convenience.

We take a partition $\mathcal{P} = \{P_1,\dots,P_K\}$ such that
$\diam(P_k, d) < \varepsilon$ and $\mu(\partial P_k)=0$ for all $1\leq k\leq K$.
 Choose a point $p_k\in P_k$ for each $k$ and set $A= \{p_1,\dots,p_K\}$.
We define a map $\mathcal{P}:\mathcal{X}\to A$ by $\mathcal{P}(P_k) = \{p_k\}$.
We have $d(x,\mathcal{P}(x)) < \varepsilon$ for all $x\in \mathcal{X}$.
For $n\geq 1$ we set $\mathcal{P}^n(x) = (\mathcal{P}(x), \mathcal{P}(Tx),\dots,\mathcal{P}(T^{n-1}x))$.

\begin{claim}    \label{claim: discretization}
   \begin{enumerate}
      \item The set $\mathcal{P}^n(S_n)$ is $4\varepsilon$-separated with respect to the distance $d_n$. 
      \item  The measure $\mathcal{P}^n_*\nu_n$ is uniformly distributed over $\mathcal{P}^n(S_n)$ and 
                $|\mathcal{P}^n(S_n)| = |S_n|$.  
   \end{enumerate}
\end{claim}

\begin{proof}
See Claim \ref{claim:  P_0^{N-1}(S_N)}.
\end{proof}

Let $0<\alpha<1/4$.
Let $X$ and $Y=(Y_0,\dots,Y_{m-1})$ be random variables such that $\mathrm{Law}(X)=\mu$, and
$Y_i$ take values in $\mathcal{X}$ and satisfy 
\[ \mathbb{E}\left(\text{the number of $0\leq i\leq m-1$ satisfying $d(T^i X, Y_i) \geq \varepsilon$}\right) < \alpha m. \]
We estimate $I(X;Y) \geq I(\mathcal{P}^m(X); Y)$ from below.
As in Remark \ref{remark: reduction to discrete case}, we can assume that the distribution of $Y$ is supported on a 
finite set $\mathcal{Y}\subset \mathcal{X}^m$.
Set $\tau = \mathrm{Law}(\mathcal{P}^m(X),Y)$, which is a probability measure on $A^m\times \mathcal{Y}$.
Since $d\left(T^i X, \mathcal{P}(T^i X)\right) < \varepsilon$, it follows that
\[ \{0\leq i\leq m-1|\, d\left(\mathcal{P}(T^i X), Y_i\right) \geq 2\varepsilon\}
    \subset \{0\leq i\leq m-1|\, d(T^i X, Y_i)\geq \varepsilon\}. \]
Thus 
\begin{equation*}
   \begin{split}
      \mathbb{E}_\tau f_{m}(x,y) &:=  \int_{A^m\times \mathcal{Y}} f_{m}(x,y) d\tau(x,y) \\
    &= \, \mathbb{E}\left(\text{the number of $0\leq i\leq m-1$ s.t. $d\left(\mathcal{P}(T^i X),Y_i\right) \geq 2\varepsilon$}\right) \\
	&< \alpha m.
   \end{split}
\end{equation*}   

For each $n\geq1$ we take a coupling $\pi_n$ of $(\mathcal{P}^m_*\mu_n, \mathcal{P}^m_*\mu)$ which minimizes 
\[  \int_{A^m\times A^m} d_m(x,y) d\pi(x,y) \]
among all couplings $\pi$ of $(\mathcal{P}^m_*\mu_n, \mathcal{P}^m_*\mu)$.
As in Claim \ref{claim: weak convergence of pi_N} in Section \ref{section: proof of main theorem},
it follows from $\mu(\partial P_k)=0$ and
Lemma \ref{lemma: convergence of optimal transport plan} in Appendix that the measures $\pi_{n_i}$
converge to $(\mathcal{P}^m\times \mathcal{P}^m)_*\mu$ in the weak$^*$ topology.
We define a coupling $\tau_n$ of $(\mathcal{P}^m_*\mu_n, \mathrm{Law}(Y))$ by composing $\pi_n$ and $\tau$:
\[  \tau_n(x,y) = \sum_{x' \in A^m} \pi_n(x,x') \mathbb{P}(Y=y|\mathcal{P}^m(X)=x'), \quad 
     (x\in A^m, y\in \mathcal{Y}). \]
$\tau_{n_i}$ converges to $\tau$ in the weak$^*$ topology.
In particular 
\begin{equation} \label{eq: expected value of f_m w.r.t. tau_n}
   \mathbb{E}_{\tau_{n_i}} f_{m}(x,y) = \int_{A^m\times \mathcal{Y}} f_{m}(x,y) d\tau_{n_i}(x,y) 
   < \alpha m  \quad 
    \text{for sufficiently large $n_i$}.
\end{equation}   
   (Here notice that $f_{m}(x,y)$ is a \textit{continuous} function on $A^m\times \mathcal{Y}$ because 
   $A^m\times \mathcal{Y}$ is a finite set.)
We define a conditional probability mass function $\tau_n(y|x)$ by 
\[ \tau_n(y|x) = \frac{\tau_n(x,y)}{\mathcal{P}^m_*\mu_n(x)}, \]
which is defined for 
\[  x\in \bigcup_{k=0}^{n-1} \mathcal{P}^m(T^k S_n), \quad y\in \mathcal{X}^m. \]

Fix a point $a\in \mathcal{X}$. 
For $n\geq 2m$, let $n=mq+r$ with $m\leq r\leq 2m-1$.
For $x\in \mathcal{P}^n(S_n)$ we define probability mass functions $\sigma_{n,i}(\cdot|x)$ $(0\leq i\leq m-1)$
on $\mathcal{X}^n$ as in (\ref{eq: definition of conditional probability functions}):
\[ \sigma_{n,i}(y|x) = \prod_{j=0}^{q-1}\tau_n\left(y_{i+jm}^{i+jm+m-1}|x_{i+jm}^{i+jm+m-1}\right) 
                       \cdot \prod_{k\in [0,i) \cup [n-r+i,n)} \delta_a (y_k). \]
 We set 
 \[  \sigma_n(y|x) = \frac{1}{m}\sum_{i=0}^{m-1}\sigma_{n,i}(y|x). \] 
Exactly as in Claim \ref{claim: lower bound on mutual information}
\begin{equation}  \label{eq: lower bound on mutual information, second main theorem}
    \frac{1}{m} I(\mathcal{P}^m_*\mu_n, \tau_n) \geq \frac{1}{n} I(\mathcal{P}^n_*\nu_n, \sigma_n).
\end{equation}

\begin{claim}  \label{claim: distortion is small, second main theorem}
We denote by $\mathbb{E}_{\mathcal{P}^{n}_*\nu_{n}, \sigma_{n}} f_{n}(x,y)$ the expected value of 
the function $f_{n}(x,y)$ 
(i.e. the number of $k \in [0,n-1]$ satisfying $d(x_k,y_k) \geq 2\varepsilon$)
 with respect to the measure 
\[  \mathcal{P}^n_*\nu_n(x) \sigma_n(y|x). \]
Then for sufficiently large $n_i$
\[ \mathbb{E}_{\mathcal{P}^{n_i}_*\nu_{n_i}, \sigma_{n_i}} f_{n_i}(x,y) < 2\alpha n_i.  \]
\end{claim}

\begin{proof}
\[ \mathbb{E}_{\mathcal{P}^{n}_*\nu_{n}, \sigma_{n}} f_{n}(x,y) = 
   \frac{1}{m} \sum_{i=0}^{m-1}\mathbb{E}_{\mathcal{P}^n_*\nu_n, \sigma_{n,i}} f_{n}(x,y). \]
\[ \mathbb{E}_{\mathcal{P}^n_*\nu_n, \sigma_{n,i}} f_{n}(x,y) \leq r + 
    \sum_{j=0}^{q-1} \mathbb{E}_{\mathcal{P}^m_* T^{i+jm}_*\nu_n, \tau_n} f_{m}(x,y). \]
Thus 
\begin{equation*}
   \begin{split}
    \mathbb{E}_{\mathcal{P}^{n}_*\nu_{n}, \sigma_{n}} f_{n}(x,y) & \leq r + 
     \frac{1}{m}  \sum_{i=0}^{m-1} \sum_{j=0}^{q-1} \mathbb{E}_{\mathcal{P}^m_* T^{i+jm}_*\nu_n, \tau_n} f_{m}(x,y) \\
    &\leq r + \frac{1}{m} \sum_{k=0}^{n-1} \mathbb{E}_{\mathcal{P}^m_*T^k_*\nu_n, \tau_n} f_{m}(x,y) \\
    &= r + \frac{n}{m} \mathbb{E}_{\mathcal{P}^m_*\mu_n, \tau_n} f_m(x,y)
       \quad (\text{by $\mu_n = \frac{1}{n}\sum_{k=0}^{n-1}T^k_*\nu_n$}) \\
    & = r + \frac{n}{m}\int_{A^m\times \mathcal{Y}}f_{m}(x,y)d\tau_n(x,y)  = r +  \frac{n}{m} \mathbb{E}_{\tau_n} f_m(x,y).
   \end{split}
\end{equation*}   
We have $\mathbb{E}_{\tau_n} f_m(x,y) < \alpha m$ for sufficiently large $n=n_i$ by (\ref{eq: expected value of f_m w.r.t. tau_n}).
Thus (by $r\leq 2m-1$)
\[  \mathbb{E}_{\mathcal{P}^{n_i}_*\nu_{n_i}, \sigma_{n_i}} f_{n_i}(x,y)
   < \alpha n_i + r < 2\alpha n_i \]
   for sufficiently large $n_i$.
\end{proof}

In view of Claim \ref{claim: discretization}, Claim~\ref{claim: distortion is small, second main theorem} and Lemma \ref{lemma: uniform distribution over separated set, second form} 
imply that for sufficiently large $n_i$,
\begin{equation}  \label{eq: lower bound on mutual information from small distortion, second main theorem}
   \frac{1}{n_i} I(\mathcal{P}^{n_i}_*\nu_{n_i}, \sigma_{n_i}) \geq \frac{1}{n_i} \log |S_{n_i}| - 2\alpha \log K - H(2\alpha).
\end{equation}

It follows 
from $|S_n|\geq \#(\mathcal{X},d_n,12\varepsilon)$ and the inequalities (\ref{eq: lower bound on mutual information, second main theorem})
and (\ref{eq: lower bound on mutual information from small distortion, second main theorem}) that
\[ \frac{1}{m} I(\mathcal{P}^m_* \mu_{n_i}, \tau_{n_i}) \geq  \frac{1}{n_i} \log \#(\mathcal{X},d_{n_i},12\varepsilon) 
    - 2\alpha \log K - H(2\alpha) \]
for sufficiently large $n_i$.
Recall that the measures $\tau_{n_i}(x,y)$ converge to $\tau = \mathrm{Law}\left(\mathcal{P}^m(X),Y\right)$.
By letting $n_i\to \infty$ we obtain that
\[  \frac{1}{m} I(\mathcal{P}^m(X); Y) \geq S(\mathcal{X}, T, d, 12\varepsilon) - 2\alpha \log K - H(2\alpha).  \]
Thus we conclude 
\[  \tilde{R}_\mu(\varepsilon, \alpha) \geq S(\mathcal{X}, T, d, 12\varepsilon) - 2\alpha \log K - H(2\alpha).  \]
Now notice that $K$ depends only on $\varepsilon$ and independent of $\alpha$.
By letting $\alpha\to 0$ we get
\[  \tilde{R}_\mu(\varepsilon) \geq S(\mathcal{X},T,d,12\varepsilon). \]
\end{proof}

Theorem \ref{thm: second main theorem} follows from Lemma \ref{lemma: upper bound, second main theorem}
and Proposition \ref{prop: lower bound, second main theorem}.

\section{Proof of Proposition \ref{prop: counter example}}  \label{section: proof of Proposition}

In this section we construct a dynamical system $(\mathcal{X},T)$ with a distance $d$ satisfying 
\begin{equation} \label{eq: requirement of counter example}
 \lim_{\varepsilon\to 0} \frac{\tilde{S}(\mathcal{X},T,d,\varepsilon)}{|\log \varepsilon|} = 0, \quad 
    \mdim_{\mathrm{M}} (\mathcal{X},T,d) = \infty. 
\end{equation}    
This proves Proposition \ref{prop: counter example} because $R_\mu(\varepsilon) \leq \tilde{S}(\mathcal{X},T,d,\varepsilon)$
 by Lemma \ref{lemma: metric mean dimension dominates rate distortion function}.

Let $V$ be an infinite dimensional Hilbert space.
We denote its norm by $\norm{\cdot}$.
We can take $A_1,A_2,\dots\subset V$ such that 
\begin{itemize}
  \item $0\in A_n$ for every $n$.
  \item For every $n$ and any two distinct points $a,b\in A_n$ we have $\norm{a-b} = 1/n$.
  \item $\log |A_n| = \Theta(2^n (\log n)^2)$, namely there exists $C>1$ independent of $n$ satisfying 
  \[  C^{-1} 2^n (\log n)^2 \leq \log |A_n| \leq C 2^n (\log n)^2. \]
\end{itemize}

Set $B = \bigcup_{n\geq 1} A_n$. 
This is a compact subset of $V$ and its diameter is bounded by $2$.
For each $n\geq 1$ we define $\mathcal{X}_n \subset A_n^\mathbb{Z}$ as the set of $(x_k)_{k\in \mathbb{Z}}$ such that 
\[  \exists l\in \mathbb{Z}: \> x_k=0 \text{ for all } k\in \mathbb{Z}\setminus \left(l+2^n\mathbb{Z}\right). \]
Set $\mathcal{X} = \bigcup_{n\geq 1} \mathcal{X}_n \subset B^\mathbb{Z}$.
This is compact with respect to the distance 
\[  d(x,y) = \sum_{k\in \mathbb{Z}} 2^{-|k|} \norm{x_k-y_k}. \]
Let $T: \mathcal{X}\to \mathcal{X}$ be the shift.
We show that $(\mathcal{X},T,d)$ satisfies 
the property (\ref{eq: requirement of counter example}).

\begin{claim} \label{claim: metric mean dimension is infinite}
\[   \mdim_{\mathrm{M}} (\mathcal{X},T,d) = \infty. \]
\end{claim}

\begin{proof}
Let $N$ be a multiple of $2^n$. For $0<\varepsilon \leq 1/n$
\[ \#(\mathcal{X}_n, d_N,\varepsilon) \geq |A_n|^{N/2^n}. \]
Thus 
\[ S(\mathcal{X}_n,T,d,\varepsilon) = \lim_{N\to \infty} \frac{1}{N} \log \#(\mathcal{X}_n, d_N,\varepsilon) 
   \geq 2^{-n} \log |A_n| = \Theta\left( (\log n)^2\right). \]
For any $0<\varepsilon<1$
\[ S(\mathcal{X},T,d,\varepsilon) \geq S(\mathcal{X}_{\lfloor 1/\varepsilon\rfloor}, T, d, \varepsilon) \geq 
   \Theta\left((\log \lfloor 1/\varepsilon\rfloor)^2\right). \]
It follows
\[ \mdim_{\mathrm{M}}(\mathcal{X},T,d) = \lim_{\varepsilon \to 0} \frac{S(\mathcal{X},T,d,\varepsilon)}{|\log \varepsilon|} 
   = \infty. \]
\end{proof}

Let $\varepsilon>0$ and set $L=L(\varepsilon) = \lceil \log_2(8/\varepsilon)\rceil$.
It follows $\sum_{|n|>L}2^{-|n|} \leq \varepsilon/4$.

\begin{claim}  \label{claim: crucial property of counter example}
  If $N\geq 2L+2^n$ and $n>\log_2 (1/\varepsilon) + \log_2(48L+24)$ then every $x\in X_n$ satisfies 
  $\bar{d}_N(x,0) < \varepsilon/2$. Here $0 = (\dots,0,0,0,\dots)\in X$.
\end{claim}

\begin{proof}
Let $x\in X_n$. There exists an integer $l$ such that $x_k=0$ for all $k\in \mathbb{Z}\setminus (l+2^n\mathbb{Z})$.
Then $d(T^i x, 0) \leq \varepsilon/4$ for any $i$ outside of $[l-L,l+L]+2^n\mathbb{Z}$. 
We count how many $i\in [0,N)$ fall in $[l-L,l+L]+2^n\mathbb{Z}$:
\begin{equation*}
  \begin{split}
   \frac{1}{N} |\left([l-L,l+L]+2^n\mathbb{Z}\right)\cap [0,N)| &\leq \frac{1}{N}\left(1+\frac{N+2L}{2^n}\right)(2L+1)\\
    &=  \left(1+\frac{2L+2^n}{N}\right) \frac{2L+1}{2^n} \\
    &\leq \frac{4L+2}{2^n} \quad (\text{by $N\geq 2L+2^n$}).
  \end{split}
\end{equation*}
Therefore 
\[ \bar{d}_N(x,0) \leq \frac{\varepsilon}{4} + \frac{3(4L+2)}{2^n} < \frac{\varepsilon}{2} \quad 
    (\text{by $n > \log_2(1/\varepsilon) + \log_2(48 L +24)$}). \]
\end{proof}

We take $\varepsilon_0>0$ so that all $0<\varepsilon<\varepsilon_0$ satisfy 
$\log_2(1/\varepsilon) > \log_2(48L(\varepsilon)+24)$.
We set $N_0(\varepsilon) = 2L(\varepsilon)+2^{\lfloor 6/\varepsilon\rfloor}$.
In the rest of this section we always assume 
\[ 0<\varepsilon < \varepsilon_0, \quad N\geq N_0(\varepsilon). \]

\begin{claim}  \label{claim: X_n is small w.r.t. d_bar}
\[  \bigcup_{n\geq 2\log_2(1/\varepsilon)} \mathcal{X}_n \subset B_{\varepsilon/2}(0,\bar{d}_N), \]
where the right-hand side is the open $\varepsilon/2$-ball around $0$ with respect to the distance $\bar{d}_N$.
Therefore 
\[ \#\left(\bigcup_{n\geq 2\log_2(1/\varepsilon)}\mathcal{X}_n, \bar{d}_N,\varepsilon\right) = 1.\]
\end{claim}

\begin{proof}
For $n> 6/\varepsilon$ every $x\in \mathcal{X}_n$ satisfies $d(x,0)\leq 3/n < \varepsilon/2$.
Thus $\mathcal{X}_n \subset B_{\varepsilon/2}(0,\bar{d}_N)$.
For $2\log_2(1/\varepsilon) \leq n \leq 6/\varepsilon$ it also follows that
$\mathcal{X}_n\subset B_{\varepsilon/2}(0,\bar{d}_N)$ by Claim \ref{claim: crucial property of counter example}
because the assumptions imply $N\geq 2L+2^n$ and $n>\log_2(1/\varepsilon) + \log_2(48L+24)$.
\end{proof}

From Claim \ref{claim: X_n is small w.r.t. d_bar} and an elementary inequality 
\[ \log (a_1+a_2+\dots+a_K) \leq \log K + \max_{1\leq i\leq K} \log a_i,  \quad (a_1,\dots,a_K>0),  \]
it follows that
\[ \log \#(\mathcal{X}, \bar{d}_N,\varepsilon) \leq 
   \log\left(1+2\log_2(1/\varepsilon)\right) + \max_{1\leq n < 2\log_2(1/\varepsilon)} 
    \log\#(\mathcal{X}_n, \bar{d}_N, \varepsilon). \]
The term $\log\#(\mathcal{X}_n, \bar{d}_N, \varepsilon)$ can be easily estimated:
\[ \#(\mathcal{X}_n,\bar{d}_N,\varepsilon) \leq \#(\mathcal{X}_n,d_N,\varepsilon) \leq 2^n |A_n|^{1+2^{-n}(N+2L)}
  \quad (\text{by $\sum_{|n|>L} 2^{-|n|} < \varepsilon/4$}). \]
\begin{equation*}
    \begin{split}
     \log \#(\mathcal{X}_n,\bar{d}_N,\varepsilon)  &\leq n\log 2 + \{1+2^{-n}(N+2L)\}  \log |A_n| \\
     &\leq n\log 2 + (2^n+N+2L) O\left((\log n)^2\right).
     \end{split}
\end{equation*}     
Hence $\log\#(\mathcal{X},\bar{d}_N,\varepsilon)$ is bounded by
\[   (2\log 2) \log_2(1/\varepsilon) + \log\left(1+ \log_2(1/\varepsilon)\right)
    + \left((1/\varepsilon)^2 + N + 2L \right) O\left((\log\log 1/\varepsilon)^2\right). \]
Thus 
\[ \tilde{S}(\mathcal{X},T,d,\varepsilon) = \lim_{N\to \infty} \frac{1}{N} \log  \#(\mathcal{X}, \bar{d}_N,\varepsilon)
   \leq   O\left((\log\log 1/\varepsilon)^2\right). \]  
So we conclude
\[ \lim_{\varepsilon\to 0} \frac{\tilde{S}(\mathcal{X},T,d,\varepsilon)}{|\log\varepsilon|} =0. \]

\appendix

\section{Elementary lemmas on optimal transport}

The purpose of this appendix is to prove lemmas on optimal transport which are used in the proofs of Theorems \ref{thm: main theorem} and 
\ref{thm: second main theorem}.
Our argument here is completely elementary.
Much more general and systematic treatments can be found in \cite{Ambrosio--Gigli--Savare} and 
\cite{Villani}.
In this appendix we identify probability measures with their probability mass functions.

Let $A$ be a finite set with a distance $d$.
For two probability measures $\mu$ and $\nu$ on $A$ we denote by $\mathscr{M}(\mu,\nu)$
the set of probability measures $\pi$ on $A\times A$ whose first and second marginals are $\mu$ and $\nu$ respectively.
We define the $L^1$-Wasserstein distance $W(\mu,\nu)$ by 
\[ W(\mu,\nu) = \min_{\pi\in \mathscr{M}(\mu,\nu)} \int_{A\times A} d(x,y) d\pi(x,y). \]
A measure $\pi \in \mathscr{M}(\mu,\nu)$ attaining this minimum is called an optimal transference plan between $\mu$ and $\nu$.

\begin{lemma}  \label{lemma: W metrizes weak* topology}
Let $\{\mu_n\}_{n\geq 1}$ be a sequence of probability measures on $A$ converging to $\mu$ in the weak$^*$ topology.
Then 
\[ \lim_{n\to \infty} W(\mu_n, \mu) = 0. \]
\end{lemma}

\begin{proof}
This is a consequence of the general fact that the Wasserstein distance metrizes the weak$^*$ topology
(\cite[Theorem 6.9]{Villani}).
Here we prove it directly.
For the notational convenience we identify $A$ with some cyclic group $\mathbb{Z}/K\mathbb{Z}$.

We define $\pi_n\in \mathscr{M}(\mu_n,\mu)$ as follows.
First we set 
\begin{equation*}
    \begin{split}
      \pi_n(0,0) &= \min(\mu_n(0), \mu(0)), \\ 
      \pi_n(0,y) &= \min\Bigl(\mu_n(0)-\sum_{k=0}^{y-1}\pi_n(0,k), \mu(y)\Bigr) \quad (1\leq y\leq K-1). 
    \end{split}
\end{equation*}    
Here we defined $\pi_n(0,y)$ inductively with respect to $y$.
Next we set 
\begin{align*}
      \pi_n(1,1) &= \min(\mu_n(1), \mu(1)-\pi_n(0,1)), \\ 
\intertext{and for $2\leq y\leq K$}
      \pi_n(1,y) &= \min\Bigl(\mu_n(1)-\sum_{k=1}^{y-1}\pi_n(1,k), \mu(y)-\pi_n(0,y)\Bigr). 
\end{align*}   
Note that $y=K$ is the same as $y=0$ in $\mathbb{Z}/K\mathbb{Z}$.
In general we set 
\begin{align*}
      \pi_n(x,x) &= \min\Bigl(\mu_n(x), \mu(x)-\sum_{k=0}^{x-1}\pi_n(k,x)\Bigr), \\ 
 \intertext{and for $x+1\leq y\leq K+x-1$}
 \pi_n(x,y) &= \min\Bigl(\mu_n(x)-\sum_{k=x}^{y-1}\pi_n(x,k), \mu(y)-\sum_{k=0}^{x-1}\pi_n(k,y)\Bigr).
\end{align*}   

The assumed convergence $\mu_n\to \mu$ in the weak$^*$ topology means that $\mu_n(x)\to \mu(x)$ for every $x$.
Then it is easy to check that 
\[ \pi_n(x,x) \to \mu(x), \quad \pi_n(x,y) \to 0 \quad (x\neq y). \]
This implies
\[ W(\mu_n,\mu) \leq \int_{A\times A} d(x,y) d\pi_n(x,y) \to 0. \]
\end{proof}

\begin{lemma}   \label{lemma: convergence of optimal transport plan}
Let $\{\mu_n\}_{n\geq 1}$ be a sequence of probability measures on $A$ converging to $\mu$ in the weak$^*$ topology.
Let $\pi_n$ be an optimal transference plan between $\mu_n$ and $\mu$.
Then the sequence $\pi_n$ converges to $(\mathrm{Id}\times \mathrm{Id})_*\mu$.
\end{lemma}

\begin{proof}
For any $a\neq b$ in $A$ 
\[ \pi_n(a,b) \leq \frac{1}{d(a,b)} \int_{A\times A} d(x,y) d\pi_n(x,y)
    = \frac{W(\mu_n,\mu)}{d(a,b)}. \]
The right-hand side converges to zero by Lemma \ref{lemma: W metrizes weak* topology}.
In the diagonal 
\[ \pi_n(b,b) = \mu(b) - \sum_{a\neq b} \pi_n(a,b) \to \mu(b). \]
\end{proof}

\vspace{0.5cm}

\vspace{0.5cm}

\address{ Elon Lindenstrauss \endgraf
Einstein Institute of Mathematics, Hebrew University, Jerusalem 91904, Israel}

\textit{E-mail address}: \texttt{elon@math.huji.ac.il}

\vspace{0.5cm}

\address{ Masaki Tsukamoto \endgraf
Department of Mathematics, Kyoto University, Kyoto 606-8502, Japan}

\textit{E-mail address}: \texttt{tukamoto@math.kyoto-u.ac.jp}

\textit{Current address}: 
Einstein Institute of Mathematics, Hebrew University, Jerusalem 91904, Israel

\end{document}